\documentclass{amsart}
\usepackage{mathtools}
\usepackage{amsmath,amssymb,amsthm,color,mathrsfs,array}
\usepackage{multirow}
\usepackage{verbatim}
\usepackage{url}

\newcommand{\re}{\mathbb{R}}

\newcommand{\N}{\mathbb{N}}

\newcommand{\A}{\mathcal{A}}

\newcommand{\lmd}{\lambda}

\newcommand{\eps}{\epsilon}

\def\af{\alpha}

\def\rank{\mbox{rank}}

\newcommand{\reff}[1]{(\ref{#1})}

\newcommand{\mc}[1]{\mathcal{#1}}
\newcommand{\mt}[1]{\mathtt{#1}}

\newcommand{\supp}[1]{\mbox{supp}(#1)}

\renewcommand{\vec}[1]{\mathit{vec}(#1)}
\newcommand{\qmod}[1]{\mbox{QM}[#1]}
\newcommand{\ideal}[1]{\mbox{Ideal}[#1]}
\newcommand{\st}{\mathit{s.t.}}
\newcommand{\hm}{\mathit{hom}}

\newcommand{\bdes}{\begin{description}}
	\newcommand{\edes}{\end{description}}
\newcommand{\bal}{\begin{align}}
	\newcommand{\eal}{\end{align}}
\newcommand{\bnum}{\begin{enumerate}}
	\newcommand{\enum}{\end{enumerate}}
\newcommand{\bit}{\begin{itemize}}
	\newcommand{\eit}{\end{itemize}}
\newcommand{\bea}{\begin{eqnarray}}
	\newcommand{\eea}{\end{eqnarray}}
\newcommand{\be}{\begin{equation}}
	\newcommand{\ee}{\end{equation}}
\newcommand{\baray}{\begin{array}}
	\newcommand{\earay}{\end{array}}
\newcommand{\bsry}{\begin{subarray}}
	\newcommand{\esry}{\end{subarray}}
\newcommand{\bca}{\begin{cases}}
	\newcommand{\eca}{\end{cases}}
\newcommand{\bcen}{\begin{center}}
	\newcommand{\ecen}{\end{center}}
\newcommand{\bbm}{\begin{bmatrix}}
	\newcommand{\ebm}{\end{bmatrix}}
\newcommand{\bmx}{\begin{matrix}}
	\newcommand{\emx}{\end{matrix}}
\newcommand{\bpm}{\begin{pmatrix}}
	\newcommand{\epm}{\end{pmatrix}}
\newcommand{\btab}{\begin{tabular}}
	\newcommand{\etab}{\end{tabular}}
\newcommand{\ter}[1]{\textcolor{red}{#1}}

\theoremstyle{plain}
\newtheorem{theorem}{Theorem}[section]
\newtheorem{pro}[theorem]{Proposition}

\newtheorem{lemma}[theorem]{Lemma}

\newtheorem*{claim*}{Claim}
\newtheorem{thm}[theorem]{Theorem}
\newtheorem{dfn}[theorem]{Definition}
\theoremstyle{definition}

\newtheorem{exm}[theorem]{Example}

\newtheorem{algorithm}[theorem]{Algorithm}

\newtheorem{remark}[theorem]{Remark}
\newtheorem{define}[theorem]{Definition}

\numberwithin{equation}{section}
\numberwithin{table}{section}

\def\rn{{\mathbb{R}^n}}
\def\r{{\mathbb{R}}}
\def\rx{{\mathbb{R}[x]}}
\def\a{{\mathcal A}}
\def\pk{{\mathscr{P}_{\mathcal{A}}(K)}}
\def\rk{{\mathscr{R}_{\mathcal{A}}(K)}}
\def\n{{\mathbb{N}}}

\def\n{{\mathbb{N}}}
\def\l{{\mathscr L}}

\begin{document}

\title{Generalized Truncated Moment Problems with Unbounded Sets}

\author[Lei Huang]{Lei Huang}
\address{Lei Huang,
	Institute of Computational Mathematics and Scientific/Engineering Computing,
Academy of Mathematics and Systems Science, Chinese Academy of Sciences,
and School of Mathematical Sciences,
University of Chinese Academy of Sciences, Beijing, China, 100049.}
\email{huanglei@lsec.cc.ac.cn}

\author[Jiawang Nie]{Jiawang~Nie}
\address{Jiawang Nie,  Department of Mathematics,
	University of California San Diego,
	9500 Gilman Drive, La Jolla, CA, USA, 92093.}
\email{njw@math.ucsd.edu}

\author[Ya-xiang Yuan]{Ya-Xiang Yuan}
\address{Ya-Xiang Yuan,
	Institute of Computational Mathematics and Scientific/Engineering Computing,
Academy of Mathematics and Systems Science,
Chinese Academy of Sciences, Beijing, China, 100049.}
\email{yyx@lsec.cc.ac.cn}



\begin{abstract}
This paper studies generalized truncated moment problems with unbounded sets.
First, we study geometric properties of the truncated moment cone
and its dual cone of nonnegative polynomials. By the technique of homogenization,
we give a convergent hierarchy of Moment-SOS relaxations
for approximating these cones. With them, we give a Moment-SOS method
for solving generalized truncated moment  problems with unbounded sets.
Finitely atomic representing measures, or certificates for their nonexistence,
can be obtained by the proposed method.
Numerical experiments and applications are also given.
\end{abstract}

\maketitle

\section{Introduction}

The generalized truncated moment problem (GTMP) concerns
whether or not there exists a positive Borel measure
which is supported in a prescribed set
and whose moments satisfy some linear equations or inequalities.
For a Borel measure $\mu$ on $\re^n$, its support
is the smallest closed set
$T \subseteq \re^n$ such that $\mu(\re^n\backslash T)=0$.
The support of $\mu$ is denoted as $\supp{\mu}$.
For a power $\alpha \coloneqq (\alpha_{1},\dots,\alpha_{n})$,
the $\af$th moment of $\mu$ is the integral
\[
\int  x_1^{\alpha_{1}} \cdots x_n^{\alpha_{n}} \mt{d} \mu
\]
if it exists. For convenience, we denote that
\[
x  \coloneqq (x_1,\dots,x_n), \quad
x^{\alpha} \coloneqq x_1^{\alpha_{1}}\cdots x_n^{\alpha_{n}}, \quad
|\alpha| \coloneqq  \alpha_{1}+\cdots+\alpha_{n}.
\]
The sum $|\af|$ is called the order of the moment $\int x^\af \mt{d} \mu$.
In applications, the support of $\mu$ is often required to
be contained in a set $K \subseteq \re^n$ such that
\begin{equation} \label{1.1}
K \coloneqq
\left\{x \in \mathbb{R}^{n} \left| \begin{array}{l}
		c_{i}(x)=0~(i \in \mathcal{E}), \\
		c_{j}(x) \geq 0~(j \in \mathcal{I})
	\end{array}\right\},\right.
\end{equation}
where all $c_i, c_j$ are polynomials in $x$.
The $\mathcal{E}$ and $\mathcal{I}$ are label sets for the constraints.
Let $\mc{A} \subset \N^n$ be a finite set of powers.
The generalized truncated moment problem
is often given as: Does there exists a  Borel measure $\mu$ such that
\[
\supp{\mu} \subseteq K, \quad  y_\af = \int x^\af \mt{d} \mu
\]
for every $\af \in \mc{A}$
and the moments $y_\af (\af \in \mc{A})$ satisfy
\be \label{bi=<aiy>}
\left\{\baray{rcl}
\sum\limits_{ \af \in \mc{A} } a_{i, \af} y_{\af}  &=& b_i \,\, (i = 1, \ldots, m_1), \\
\sum\limits_{ \af \in \mc{A} } a_{i, \af} y_{\af}  &\ge&  b_i \,\, (i = m_1+1, \ldots, m) ?
\earay \right.
\ee
In the above, $a_{i, \af}, b_i$ are given constants.
For each $i$, let
\[
a_i(x) \, = \,
\sum\limits_{ \af \in \mc{A} } a_{i, \af} x^{\af} .
\]
Then \reff{bi=<aiy>} is equivalent to
the existence of a Borel measure $\mu$ supported in $K$ such that
\be  \label{aixmu>=bi}
\left\{\baray{rcll}
\int a_i(x) \mt{d} \mu  &=& b_i & (i = 1, \ldots, m_1), \\
\int a_i(x) \mt{d} \mu  &\ge&  b_i & (i = m_1+1, \ldots, m) .
\earay \right.
\ee

How do we determine the existence of a  Borel measure
$\mu$ supported in $K$ and satisfying \reff{aixmu>=bi}?
The GTMP is a fundamental question in optimization
\cite{hkl20,de,LasBk15,laurent2009sums}.
It has broad applications (see \cite{2014The,Nie2015Linear,tang}). The generalized truncated moment optimization problem is studied in
\cite{2008A,Nie2015Linear}.
In applications, we are often interested in finitely atomic measures.
The measure $\mu$ is called finitely atomic
if the support $\supp{\mu}$ is a finite set.
It is called $r$-atomic
if $\supp{\mu}$ consists of $r$ distinct points,
i.e., the cardinality $|\supp{\mu}|=r$.

\subsection{The $\a$-truncated $K$-moment problem}

A special case of the GTMP is   the $\a$-truncated $K$-moment problem ($\a$-TKMP).
For a finite power set $\mc{A} \subseteq \n^n$, we denote by $\re^{\a}$
the set of all real vectors $y$ labelled as
\[
 y =  (y_\af)_{ \af \in \a}.
\]
A vector $y \in \re^{\A}$ is  called an $\A$-truncated multi-sequence ($\a$-tms).
The $\A$-tms $y$ is said to {\it admit}
a Borel measure $\mu$ on $\re^n$ if it holds that
\be  \label{y-meas}
y_\af \, = \, \int x^\af \mt{d} \mu \quad
\mbox{for all} \quad \af \in \A.
\ee
If it satisfies \reff{y-meas}, the measure $\mu$ is called a
{\it representing measure} for $y$
and $y$ is called an $\A$-truncated moment sequence.
The $\A$-truncated moment problem  concerns
the existence and computation of a representing measure $\mu$.
If $\supp{\mu} \subseteq K$, the measure $\mu$ satisfying \reff{y-meas}
is called a $K$-representing measure for $y$.
Denote by $meas(y,K)$ the set of all $K$-representing measures for $y$.
This gives the $\A$-truncated moment cone
\be \label{meas:K-d}
\mathscr{R}_\A(K) \, \, \coloneqq \, \,
\big\{y \in \re^{ \A }:  \,meas(y,K) \ne \emptyset \big\}.
\ee
For the case $y = 0$, its representing measure
is the identically zero measure and the support is the empty set.
Let $\r[x]_{\a}$ be the subspace spanned by monomials $x^{\af}$ ($\alpha \in \a$).
The dual cone of $\mathscr{R}_\A(K)$ is the cone of
polynomials in $\mathbb{R}[x]_{\a}$ that are nonnegative on $K$, i.e.,
\[
\mathscr{P}_{\a}(K)  \, =  \,
\{p \in \rx_{\a} : p(x)\geq 0, \, \forall\,  x \in K \}.
\]
When the polynomials $a_i$ are the monomials $x^\af$ $(\alpha \in \a)$,
the GTMP is reduced to the $\a$-TKMP (see \cite{2014The}).
When $\A$ is the power set (for a degree $d$)
\[
\n^n_d  \, \coloneqq \, \{\alpha \in \n^n:\, |\alpha| \leq d \},
\]
the GTMP is reduced to the classical truncated $K$-moment problem
(see \cite{curto2005truncated}).

\subsection{PSOP and SCP tensors}

Truncated moment problems have broad applications in tensor computation
(see \cite{Nuclear17,NYZ18,NieZhang18}).
For positive integers $m,n$, let $\mathrm{T}^{m}\left(\mathbb{R}^{n}\right)$
denote the space of all order-$m$ tensors over the space $\re^n$.
A tensor $\mathcal{B} \in \mathrm{T}^{m}\left(\mathbb{R}^{n}\right)$
is represented by an $m$-way array
$
\mathcal{B}=\left(\mathcal{B}_{i_{1} \ldots i_{m}}
\right)_{1 \leq i_{1}, \ldots, i_{m} \leq n}
$,
with real entries $\mathcal{B}_{i_{1} \ldots i_{m}}$.
The tensor $\mathcal{B}$ is said to be  symmetric if
the entry $\mathcal{B}_{i_{1} \ldots i_{m}}$ is invariant
under all permutations of $\left(i_{1}, \ldots, i_{m}\right)$. Let $\mathrm{S}^{m}\left(\mathbb{R}^{n}\right)$ be the subspace of
symmetric tensors in $\mathrm{T}^{m}\left(\mathbb{R}^{n}\right)$.
For $u \in \mathbb{R}^{n}$, the outer product $u^{\otimes m}$ denotes  the tensor in $\mathrm{S}^{m}\left(\mathbb{R}^{n}\right)$ such that
for every $(i_{1}, \ldots, i_{m})$,
\[
\left(u^{\otimes m}\right)_{i_{1} \ldots i_{m}}=u_{i_{1}} \cdots u_{i_{m}}.
\]

\begin{define} \label{def:PSOP}
A symmetric tensor $\mathcal{B} \in \mathrm{S}^{m}(\mathbb{R}^{n+1})$
is said to be a positive sum of powers (PSOP) tensor if there exist vectors
$v_{1}, \ldots, v_{r} \in \re^{n}$ and scalars
$\lambda_{1},\dots,\lambda_{r}\geq 0$ such that
\be \label{pnsop}
\mathcal{B} = \sum_{k=1}^{r}\lambda_{k}
\bbm 1\\ v_k \ebm^{\otimes m}.
\ee
If all $v_{1},\dots v_{r}$ are nonnegative vectors
(i.e., all their entries are nonnegative),
then the tensor $\mc{B}$ as in \reff{pnsop}
is said to be strongly completely positive (SCP).
\end{define}

Denote by $\mathrm{PS}^{n+1}_m$ the set of all PSOP tensors in
$\mathrm{S}^{m}(\mathbb{R}^{n+1})$.
Similarly, the set of all SCP tensors in
$\mathrm{S}^{m}(\mathbb{R}^{n+1})$
is denoted as $\mathrm{SCP}^{n+1}_m$.
Every symmetric tensor in $\mathrm{S}^{m}(\mathbb{R}^{n+1})$
is uniquely determined by an $\a$-tms with $\a = \N^n_m.$
More specifically, each symmetric tensor $\mathcal{B}=\left(\mathcal{B}_{i_{1} \ldots i_{m}}
\right)_{0 \leq i_{1}, \ldots, i_{m} \leq n}$  
is uniquely determined by $\mathbf{b} = (b_\af) \in \re^{\a}$ such that
\be \label{tensortms}
b_{\alpha} \, =  \, \mathcal{B}_{i_{1} i_{2} \ldots i_{m}}
\ee
for every $\af \in \N^n_m$ with $x_0^{m-|\alpha|}\cdot x^\alpha
= x_{i_{1}}  \cdots x_{i_{m}}$.
Then, one can see that the decomposition \reff{pnsop} is equivalent to
\be \label{tensormom}
b_{\alpha}  =   \int x^{\alpha} d \mu \,\,
(\forall\, \alpha \in \n^{n}_{m} ),
\ee
where $\mu=\sum_{k=1}^{r}\lambda_{k}\delta_{v_k}$ and $\delta_{v_k}$
denotes the unit Dirac measure supported at $v_k$.
This shows that the tensor $\mathcal{B}$ is PSOP if and only if  $\mathbf{b}$ admits a representing measure on $\rn$.
Similarly, the tensor $\mathcal{B}$ is SCP
if and only if $\mathbf{b}$ admits a representing measure supported in
the nonnegative orthant $\r_+^n$.

When the order $m$ is even, the tensor $\mc{B}$ as in \reff{pnsop}
is a sum of even powers (SOEP) tensor (see \cite{Nuclear17}).
When every $v_k$ is nonnegative, $\mc{B}$ is a completely positive (CP) tensor.
For nonzero PSOP and SCP tensors, the first entry of each decomposing vector
is required to be strictly positive.
Detecting SOEP and CP tensors is equivalent to a homogeneous
truncated moment problem with the set $K$ being the unit sphere \cite{2014The}.
However, this is not the case for PSOP and SCP tenors.
We remark that every SCP tensor is a CP tensor,
but not every CP tensor is SCP.
We refer to \cite{NTYZ22} for introductions to CP tensors.


\subsection{Existing work}

To solve truncated moment problems, a typical method
is to use flat extension (see Section~\ref{Rieszfunc} for the introduction).
We refer to
\cite{curto2002solution,curto2005solution,curto2013recursively,fialkow2008truncated}
for related work.
However, this approach has difficulty for general cases,
especially when the  moment matrix is positive definite.
When the set $K$ is compact, Moment-SOS relaxations can be used to
solve truncated moment problems (see \cite{2008A,2014The,Nie2015Linear}).
When a tms $y$ admits a $K$-representing measure,
it can be obtained by finding a flat extension.
When it admits no $K$-representing measures,
a certificate for the nonexistence can also be obtained
by solving Moment-SOS relaxations.
There are other alternate approaches for studying GTMP.
For instance, the core variety can be used to characterize
existence of representing measures
(see \cite{blekherman2018core,fialkow2017core}).

When $K=\rn$, the homogenization trick is proposed in \cite{Fialkow2012The}
to solve the truncated moment problem for the case $\mc{A} = \N^n_d$.
Homogenization is a useful trick for solving
optimization with unbounded sets \cite{hny21,nie2012discriminants}.
When $K$ is an unbounded semialgebraic set,
there exists little work for solving truncated moment problems.

\subsection*{Contributions}

In this paper, we use homogenization to solve the GTMP
when $K$ is an unbounded set. The main approach is lifting $K$
to a compact semialgebraic set in $\r^{n+1}$ via homogenization.
Then a hierarchy of Moment-SOS relaxations is proposed
to solve the GTMP.  Throughout the  paper, we assume the feasible set
$K$ is closed at $\infty$ (see Definition \ref{closed:inf}).
Our major contributions are:

\bit

\item[I.]
We study  geometric properties of the truncated moment cone $\rk$
and its dual cone $\pk$ of nonnegative polynomials.
A convergent hierarchy of
Moment-SOS approximations is given for $\pk$ and $\rk$.
The convergence is shown under some general assumptions.
\par

\item[II.]
We propose a hierarchy of Moment-SOS relaxations to solve the  GTMP
when $K$ is an unbounded semialgebraic set.
When the closure of the moment system (see \reff{eqq5.6})
admits no $K$-representing measures, we show that
the moment relaxations must be infeasible for all high enough relaxation orders.
When the closure of the moment system is feasible,
the method can obtain a finitely atomic representing measure.
Depending whether or not the first entries of points in the support
are zero, we can either obtain a finitely atomic $K$-representing measure
for the original GTMP or obtain arbitrarily accurate approximations.

\eit

This paper is organized as follows.
Section~\ref{sc:pre} gives some preliminaries.
In Section~\ref{sc:geo}, we give geometric properties of
the cones $\pk$ and $\rk$. Moment-SOS relaxations are also given for them.
In Section~\ref{Generalizedtkmp}, a Moment-SOS algorithm
is given to solve the GTMP. Its convergence is also shown.
Some numerical experiments and applications are presented in Section~\ref{sc:num}.
Some conclusions are given in Section~\ref{sc:con}.

\section{Preliminaries}
\label{sc:pre}

\subsection*{Notation}

Let $\n$ (resp., $\r$) denote the set of nonnegative integers
(resp., real numbers). For $t \in \r$, $\left\lceil t\right\rceil$
denotes the smallest integer greater than or equal to $t$.
For a positive integer $m$, denote $[m] \coloneqq \{1, \ldots, m\}$.
Denote by $\rx$  the ring of real polynomials in $x$.
Let $\rx_d$ be the set of all polynomials with degrees $\leq d$.
For a set $S \subseteq \rn$, $cl(S)$ and $int(S)$
denote its closure and interior  in the Euclidean topology.
The dual cone of $S$ is the set
\[
S^* \, \coloneqq \,
\{y\in \rn \mid  \langle x,y\rangle \geq 0,~ \forall x \in S\}.
\]
When $S$ is a convex set, $ri(S)$
denotes its relative interior.
For a set $V$ of vectors,
$\mbox{Span}\, V$ stands for the subspace spanned by $V$
and $\mbox{Cone}\,V$ stands for the conic hull
generated by vectors in $V$. The notation $\mathrm{dim}\, V$
denotes the dimension of affine  hull of $V$.
We refer to 
\cite{Bert} for these notions.

For a polynomial $p \in \re[x]$, $\deg(p)$ denotes its total degree in $x$ and
$p^\hm$ denotes the homogeneous part of the highest degree terms of $p$.
Denote by $\tilde{p}$  the homogenization of $p$, i.e., $\tilde{p}\coloneqq x_0^{\deg(p)}p(\frac{x}{x_0})$.
For a finite power set $\mc{A}$ of $\n^n$, the degree of $\A$,
denoted by  $\deg(\A)$, is $ \max\{\left|\alpha\right|: \af \in \a\}$.
Let $[x]_{\a}$ be the column vector of monomials $\left(x^{\alpha}\right)_{\alpha \in \a}$.
For a set $K \subseteq \re^n$,
denote by $I(K)_{\a}$  the set of all polynomials in $\r[x]_{\a}$
which identically vanish on $K$, i.e.,
\be \label{def:I(K)A}
I(K)_{\a} \, = \, \{p\in \r[x]_{\a}: p\left|_{ K}\right.  \equiv 0\}.
\ee
Clearly, $I(K)_{\a}$ is a subspace of $\rx_{\a}$.
When $\a=\n^n_d$, we denote that
\[
\mathscr{P}_{d}(K) \, \coloneqq \,  \mathscr{P}_{\a}(K), \quad
\mathscr{R}_{d}(K) \, \coloneqq \,  \mathscr{R}_{\a}(K) .
\]

\subsection{Riesz functionals and  moment matrices}
\label{Rieszfunc}

An $\a$-tms $y \in \re^{ \mc{A} }$ determines the Riesz functional
$\mathscr{L}_{y}$ acting on $\rx_{\a}$ as
\begin{equation*}
\mathscr{L}_{y}(\sum_{\alpha \in \a} p_{\alpha} x^{\alpha})
\,  \coloneqq  \, \sum_{\alpha \in \a} p_{\alpha} y_{\alpha}.
\end{equation*}
For convenience of notation, we also write that
\begin{equation*}
\langle \sum_{\alpha \in \a} p_{\alpha} x^{\alpha}, y \rangle
\,  = \, \sum_{\alpha \in \a} p_{\alpha} y_{\alpha}.
\end{equation*}
The functional $\l_y$ is said to be $K$-positive if
\begin{equation*}
	\l_y(p) \geq 0, ~\forall~ p \in \mathscr{P}_{\a}(K).
\end{equation*}
Furthermore, $\l_y$ is said to be strictly $K$-positive if
\begin{equation*}
	\l_y(p) > 0, ~\forall~ p \in \mathscr{P}_{\a}(K),~
p\left|_{ K}\right. \not \equiv 0.
\end{equation*}
Clearly, if $y$ admits a $K$-representing measure $\mu$,
then $\l_y$ must be $K$-positive, since
$\l_y(p) = \int p(x) \mt{d} \mu \geq 0$
for all $p \in \mathscr{P}_{\a}(K)$.
When $K$ is compact and $\A = \n^n_t$,
it is known that $y \in \mathbb{R}^{\mathbb{N}_{t}^{n}}$ admits a
$K$-representing measure if and only if
the Riesz functional $\mathscr{L}_{y}$ is $K$-positive
(see \cite{laurent2009sums,GTMP}).
This conclusion can be generalized to
any finite-dimensional subspace of $\rx$.

\begin{thm}[\cite{Fialkow2012The}] \label{thm2.1}
Let $K \subseteq \mathbb{R}^{n}$ be a compact set
and let $H$ be a finite dimensional subspace of $\mathbb{R}[x]$.
Suppose there exists $f  \in H$ such that $ f >0$ on $K$.
For a linear functional $\mathscr{L}: H \rightarrow \mathbb{R}$,
if $\mathscr{L}$ is $K$-positive, i.e.,
\[
p \in H,\left.\quad p\right|_{K} \geqslant 0
\quad \Rightarrow \quad \mathscr{L}(p) \geqslant 0,
\]
then there exist $m \leqslant \dim H, \, u_{1}, \ldots, u_{m} \in K,$
and $\lambda_{1}, \ldots, \lambda_{m}>0,$ such that
\be \label{meas:pinH}
 \mathscr{L}(p)=\sum_{i=1}^{m} \lambda_{i} p\left(u_{i}\right) \quad
 (\forall \, p \in H).
\ee
\end{thm}

In Theorem \ref{thm2.1}, the equation~\reff{meas:pinH} is equivalent to that
$\mathscr{L}$ admits a finitely atomic measure supported in $K$.
When $K$ is compact and $H = \re[x]_\A$,
solving $\A$-TKMP essentially requires to check
whether $\l_{y}$ is $K$-positive or not.
However, checking the $K$-positivity is typically a hard question
\cite{fialkow2014truncated}.

To solve truncated moment problems, we review
moment and localizing matrices.
For $y \in \r^{\n^n_{2k}}$ and $q \in \rx_{2k}$,
the $k$th order localizing  matrix of $y$ generated by $q$
is the symmetric matrix $L_q^{(k)}$ such that
\[
\mathscr{L}_{y}\left(q p^{2}\right)   =
\vec{p}^{T}\left(L_{q}^{(k)}[y]\right) \vec{p},
\]
for all $p \in \mathbb{R}[x]$ with
$\operatorname{deg}\left(q p^{2}\right) \leq 2 k$.
Here, $\vec{p}$ denotes the coefficient vector of $p$.
In particular, when $q=1$,  $M_k[y] \coloneqq L_q^{(k)}[y]$
is called the $k$th order moment matrix of $y$, which satisfies
\begin{equation*}
\mathscr{L}_{y}\left(p^{2}\right) \,  = \, \vec{p}^{T} M_{k}[y] \vec{p}
\end{equation*}
for all $p \in \mathbb{R}[x]_{k}$.
Let $K$ be as in $(\ref{1.1})$. Denote the degree
\be \label{deg:dK}
d_K  \coloneqq  \max_{i \in \mc{E} \cup \mc{I} }
\left\{1,\left\lceil\operatorname{deg}\left(c_{i}\right) / 2\right\rceil  \right\}.
\ee
If $y$ admits a $K$-representing measure $\mu$, then
$\mathscr{L}_{y}(q p^{2}) = \int q p^2 \mt{d} \mu$
for all $p,q$ with the degree $\deg(qp^2) \le 2k$.
This implies that
\begin{equation} 	\label{1.2}
M_{k}[y] \succeq 0, ~	L_{c_i}^{(k)}[y] = 0~ (i\in \mathcal{E}), ~
L_{c_j}^{(k)}[y]  \succeq 0~(j \in \mathcal{I}).
\end{equation}
The condition $(\ref{1.2})$ is necessary for $y$ to admit a
$K$-representing measure, but it  may not be sufficient.
In computational practice, a convenient condition
is the flat extension.
The tms $y\in \r^{\n^n_{2k}}$ is said to be flat if it satisfies the rank condition
\begin{equation}  	\label{1.3}
\operatorname{rank} M_{k-d_{K}}[y]
\, = \,  \operatorname{rank} M_{k}[y].
\end{equation}
The following is the classical result about flat extension
by Curto and Fialkow \cite{curto2005truncated}.

\begin{thm}[\cite{curto2005truncated}]  \label{thm1.1}
Let $K$ be as in $(\ref{1.1})$. If $y \in \mathbb{R}^{\mathbb{N}_{2 k}^{n}}$
satisfies \reff{1.2} and it is flat,
then $y$ admits a unique $K$-representing measure,
which is  $\operatorname{rank} M_{k}[y]$-atomic.
\end{thm}

For the case $y \in \mathbb{R}^{\a}$, $z \in \mathbb{R}^{\mathbb{N}_{2k}^{n}}$
and $2k \geq \deg(\a)$, the $z$ is said to be an extension of $y$ if
$y$ is the subvector of $z$ consisting of entries
labelled by $\af \in \mc{A}$. For such a case, we write that
$y=\left.z\right|_{ \a}$.  Indeed, it is shown in \cite{curto2005truncated}
that $y \in \mathbb{R}^{\mathbb{N}_{t}^{n}}$ admits a $K$-representing measure
if and only if it has a flat extension $z \in \mathbb{R}^{\mathbb{N}_{2 k}^{n}}$
that satisfies \reff{1.2} for some $2k\geq t$.

\subsection{Positive polynomials}

A set $I  \subseteq \mathbb{R}[x]$ is called an ideal
if $I \cdot \mathbb{R}[x] \subseteq \mathbb{R}[x]$ and $I+I \subseteq I$.
For a polynomial tuple $h=(h_1,\dots, h_t)$,
its real variety is the set
\[
V_{\r}(h)  \, \coloneqq \,
\{x\in \r^n \mid h_1(x)=\cdots=h_t(x)=0\}.
\]
The ideal generated by $h$ is denoted  by $\ideal{h}$, i.e.,
\begin{equation*}
\ideal{h}  \, \coloneqq \, h_1 \cdot \mathbb{R}[x]+\cdots+h_t \cdot \mathbb{R}[x].
\end{equation*}
For a degree $k$, the $k$th truncated ideal of $\ideal{h}$ is
\be  \label{ideal[h]k}
\ideal{h}_k  \, \coloneqq \, h_1 \cdot \mathbb{R}[x]_{k-\deg(h_1)}+\cdots+h_t \cdot
\mathbb{R}[x]_{k-\deg(h_t)}.
\ee

A polynomial is said to be SOS if $p=p_1^2+\dots+p_r^2$ for
$p_1,\dots,p_r \in \mathbb{R}[x]$.
Denote  the set of all SOS polynomials by $\Sigma[x]$.
For an even degree $k\in \n$, denote the truncation
\[
\Sigma[x]_{k} \, \coloneqq \, \Sigma[x] \cap  \mathbb{R}[x]_{k}.
\]
We refer to \cite{HillarNie08,Las01,PMI11,SOS12} for SOS polynomials and matrices.
For a polynomial tuple $g=(g_1,\dots,g_{\ell})$,
the quadratic module generated by $g$ is
\begin{equation}
\qmod{g}   \, \coloneqq \, \Sigma[x]+ g_1 \cdot \Sigma[x]+
\cdots+ g_{\ell} \cdot \Sigma[x].
\end{equation}
For an even degree $k$, its $k$th truncated quadratic module is defined as
\begin{equation}
\qmod{g}_k  \, \coloneqq \,
\Sigma[x]_{k}+ g_1 \cdot \Sigma[x]_{k-2\lceil \deg(g_1)/2\rceil}+\cdots+
g_{\ell} \cdot \Sigma[x]_{k-2\lceil \deg(g_\ell)/2\rceil}.
\end{equation}

The sum $\ideal{h}+\qmod{g} $ is said to be archimedean
if there exists $N >0$ such that $N-\|x\|^2 \in \ideal{h}+\qmod{g} $. Let
\[
S=\{x\in \rn:h(x)=0,~g(x)\geq 0\}.
\]
Clearly, if $\ideal{h}+\qmod{g}$ is archimedean,
then $S$ is a compact set.
Note that if $p \in \ideal{h}+\qmod{g}$, then $p \ge 0$ on $S$.
However, the inverse may not be true.
The following conclusion is usually referenced as Putinar's Positivstellensatz.

\begin{thm}[\cite{putinar1993positive}]  	\label{thm2.2}
Suppose $\ideal{h}+\qmod{g}$ is  archimedean.
If $p$ is positive on  $S$,
then $p \in  \ideal{h}+\qmod{g}$.
\end{thm}

We refer to \cite{Las01,LasBk15,laurent2009sums,nieopcd}
for more introductions to moment and polynomial optimization.

\section{Geometric properties of $\rk$ and $\pk$}
\label{sc:geo}

This section gives geometric properties of
the moment cone $\rk$ and the nonnegative polynomial cone $\pk$.
These properties include relative interiors, closedness and duality.
Moreover, we give convergent Moment-SOS relaxations for $\rk$ and $\pk$.
The set $K$ is not assumed to be compact.

%
%


The followings are basic properties of the cones $\pk$ and $\rk$.
Recall that $I(K)_{\a}$ denotes  the vanishing ideal
 as in \reff{def:I(K)A}.

\begin{thm} \label{thm3.1}
Let $\a \subseteq \n^n$ be a finite power set,
$d= \deg(\A)$
and let $K$ be the semialgebraic set as in \reff{1.1}.
Then, we have
\bit
	
\item [(i)] It holds the duality relation
\be \label{dual:PARA(K)}
\mathscr{R}_{\a}(K)^*  = \mathscr{P}_{\a}(K), \quad
\mathscr{P}_{\a}(K)^*  =  cl(\mathscr{R}_{\a}(K)).
\ee
If $K$ is compact and
there exists $q \in \rx_{\a}$ such that $q >0$ on $K$, then
\[
int(P_{\a}(K))=\{p\in \r[x]_{\a}:p\left|_{ K}\right. >0\},
\]
and the moment cone $\mathscr{R}_{\a}(K)$ is closed.

\item [(ii)]
The subspace spanned by the cone $\mathscr{R}_{\a}(K)$ is
\be \label{aff:RA(K)}
\mbox{Span} \, \mathscr{R}_{\a}(K)  \, = \,
\{y \in \r^{\a} \mid \langle p, y \rangle = 0,
\,\, \forall \,  p \in I(K)_{\a} \},
\ee
or equivalently, $\mbox{Span} \, \mathscr{R}_{\a}(K)  = I(K)_{\a}^\perp$.
This implies that
\[
\dim \mathscr{R}_{\a}(K) \, = \,
|\A|-\mathrm{dim}\, I(K)_{\a} .
\]
Furthermore, the relative interior of
$\mathscr{R}_{\a}(K)$ is given as
\be \label{ri:RA(K)}
ri\left(\mathscr{R}_{\a}(K)\right)  =
\{y \in \r^{\a} \mid \mathscr{L}_{y} ~
\text{is strictly}~ K\text{-positive} \}.
\ee
In particular, if $I(K)_{\a} = \{ 0 \}$, the moment cone
$\mathscr{R}_{\a}(K)$ has nonempty interior.


\eit
\end{thm}
\begin{proof}
	
(i) By the definition, one can directly check that
$\mathscr{R}_{\a}(K)^*  =  \mathscr{P}_{\a}(K).$
Note that $\mathscr{R}_{\a}(K)$ is a convex cone.
Clearly, its closure is a closed convex cone.
By the bi-duality theorem (see \cite{Bert}), it holds that
\[
cl(\mathscr{R}_{\a}(K))
=  \big(  cl(\mathscr{R}_{\a}(K) )^* \big)^*
=  ( \mathscr{R}_{\a}(K)^* )^*
=   \mathscr{P}_{\a}(K)^*.
\]
Therefore, the duality relation \reff{dual:PARA(K)} holds.

When $K$ is compact and there exists $q \in \rx_{\a}$ such that $q >0$ on $K$,
the closedness of $\rk$ follows from Theorem \ref{thm2.1}.
For each $p>0$ on $K$, we have $p+h>0$ on $K$ for every
$h \in \mathbb{R}[x]_{\mathcal{A}}$ with sufficiently small coefficients.
This means that $p \in int\left(\mathscr{P}_{\mathcal{A}}(K)\right)$.
Conversely, for each $p \in \operatorname{int}\left(\mathscr{P}_{\mathcal{A}}(K)\right)$,
we have $p-\epsilon q \in \mathscr{P}_{\mathcal{A}}(K)$ for some $\epsilon>0$.
Then $p \geq \epsilon q>0$ on $K$.

(ii) First, we prove \reff{aff:RA(K)}.
Note that $\mbox{Span} \, \mathscr{R}_{\a}(K) $ is the subspace
spanned by all vectors $[u]_{\a}$, with $u \in K$.
For every $u \in K$ and for every $p \in I(K)_{\a}$, it is clear that
$\langle [u]_{ \a},p\rangle =0$, so
\[
\mbox{Span} \, \mathscr{R}_{\a}(K)  \, \subseteq \,
\{y \in \r^{\a} \mid \langle p, y \rangle = 0,
\,\, \forall \,  p \in I(K)_{\a} \} .
\]
We prove the reverse containment ``$\supseteq$" also holds.
For each $\xi \in \r^{\a}$ such that
$\xi \not\in \mbox{Span} \, \mathscr{R}_{\a}(K)$,
there must exist a polynomial $q \in \re[x]_{\a}$ such that
\[
\langle q, \xi \rangle = 1, \quad
\langle q, [u]_{\a} \rangle = 0, \,\, \forall \, \, u \in K.
\]
This implies that
\[
\mbox{Span} \, \mathscr{R}_{\a}(K)  \, \supseteq \,
\{y \in \r^{\a} \mid \langle p, y \rangle = 0,
\,\, \forall \,  p \in I(K)_{\a} \} .
\]
So, the equality \reff{aff:RA(K)} holds and hence
\[
\mathrm{dim}\, \mathscr{R}_{\a}(K)  \, = \,
\left| \a \right|-\mathrm{dim}\, I(K)_{\a}.
\]

Second, we prove \reff{ri:RA(K)}.
Since $K$ is closed, there exists a positive Borel measure $\mu$
whose support is exactly equal to $K$.
Let $z$ be the $\a$-tms such that
$z_{\alpha}=\int_K x^{\alpha} \mathrm{d}\mu $
for every $ \alpha \in \a $.
Clearly, $\l_z$ is strictly $K$-positive and $z \in \mbox{Span} \, \mathscr{R}_{\a}(K)$.
The equality \reff{ri:RA(K)} is implied by the following two facts.

1) For each $y \in ri\left(\mathscr{R}_{\a}(K)\right)$,
there exists $\epsilon >0$ such that $y-\epsilon z \in \mathscr{R}_{\a}(K)$, so
\[
\mathscr{L}_{y}(p) =  \mathscr{L}_{y - \eps z}(p)
+ \eps \mathscr{L}_{z}(p) \ge  \eps \mathscr{L}_{z}(p) > 0
\]
for every $p \in \pk$, $p\left|_{ K}\right. \not\equiv 0$.
This means that $\mathscr{L}_{y}$ is strictly $K$-positive.

2) Consider $y\in \r^{\a}$ such that $\mathscr{L}_{y}$ is strictly $K$-positive.
Let  $P:=\mathbb{R}[x]_{\a} / I(K)_{\a}$ be the quotient space.
The equivalent class of $p \in \re[x]_{\a}$ in $P$ is denoted as $[p]$.
Its quotient norm is
\[
\|[p]\|_P  \, = \, \inf_{a \in I(K)_{\a} } \|p - a\|_2 .
\]
Define the Riesz functional $\tilde{\mathscr{L}_y}$ on $P$ such that
\[
\tilde{\mathscr{L}_y}([p]) \, \coloneqq \,  \mathscr{L}_{y}(p) .
\]
For $q \in I(K)_{\a}$, we have
$\mathscr{L}_{y}(q)\geq 0$ and $\mathscr{L}_{y}(-q)\geq 0$,
so $\mathscr{L}_{y}(q)=0$.
This means that $\tilde{\mathscr{L}_y}$ is well-defined.
Denote the set
\[
T  \coloneqq  \{ [p] \in \pk / I(K)_{\a}:  \|[p]\|_P=1 \}.
\]
Clearly, $T$ is a compact subset of $P$.
Denote the constants
\[
\baray{rcl}
\delta_1  & \coloneqq  &
\min \Big\{\tilde{\mathscr{L}}_{y}([p]):[p]\in T \Big \},  \\
\delta_2  & \coloneqq &
\max \left\{ \tilde{\mathscr{L}}_{w}([p])
\left| \baray{l}
w \in \mbox{Span} \, \mathscr{R}_{\a}(K), \\
\| w \| = 1, \, [p]\in T
\earay \right.
\right\}.
\earay
\]
Note that $\mathscr{L}_{y}$ is  strictly $K$-positive
if and only if $\delta_1 > 0$.
For every $w \in \mbox{Span} \, \mathscr{R}_{\a}(K)$ with $\| w \| = 1$
and for every $0 \le \eps \le \frac{\delta_1}{2\delta_2}$, it holds that
\[
\tilde{\mathscr{L}}_{y + \eps w} ([p]) =
\tilde{\mathscr{L}_{y } }([p])  + \eps \tilde{\mathscr{L}_{w} }([p])
\ge \delta_1-\frac{\delta_1}{2} > 0
\]
for all $[p] \in T$. Thus, $\mathscr{L}_{y + \eps w}$ is $K$-positive and
$y + \eps w \in  cl(\mathscr{R}_{\a}(K))$.
Note that $ri(cl(\rk))=ri(\rk)$.
This means that $ y \in ri\left(\mathscr{R}_{\a}(K)\right).$

Last, if $I(K)_{\a} = \{ 0 \}$, then $\mathrm{dim}\, \mathscr{R}_{\a}(K)=\left| \a \right|$,
i.e., $\mathscr{R}_{\a}(K)$ is full-dimensional
and it must have nonempty interior.
\end{proof}

As an exposition, consider the PSOP tensor cone $\mathrm{PS}^{n+1}_m$
(see Definition~\ref{def:PSOP}). Observe that
$\mathrm{PS}^{n+1}_m = \mathscr{R}_{m}(K)$
with $K = \re^n$.
Since $I(\r^n)_{\n^n_m}= \{0\}$, we know the cone $\mathrm{PS}^{n+1}_m$
is full-dimensional and has nonempty interior,
by Theorem~\ref{thm3.1}.
When $K$ is unbounded, the moment cone $\mathscr{R}_{\a}(K)$
is generally not closed, while $\mathscr{P}_{\a}(K)$ is always closed.
For instance, when $K=\r$ and $\a=\n^1_4$, for  $y=\left(1,1,1,1,2\right)$,
the Riesz functional $\l_y$ is $K$-positive but $y$ admits no $K$-representing measures (see \cite{Ra2008An}).
In \cite{easwaran2011positive}, Easwaran and Fialkow exhibited
a family of positive linear functionals
that admit no representing measures.
However, if $K$ is compact and there exists $q \in \rx_{\a}$ such that $q >0$ on $K$,
the cone $\mathscr{R}_{\a}(K)$ is closed
(see item (i) of Theorem \ref{thm3.1}).
When $K$ is compact, convergent semidefinite relaxations for $\pk$ and $\rk$
are given in \cite{Nie2015Linear}.
When $K$ is unbounded, these relaxations typically do not converge.
In the following, we use the trick of homogenization
to construct convergent semidefinite relaxations for $\pk$ and $\rk$
when $K$ is unbounded.

\subsection{Homogenization}

Let $K$ be the set as in \reff{1.1} and let $\tilde{x} \, \coloneqq \,  (x_0,x)$.
For a polynomial $p$, recall that  $\tilde{p} = x_0^{\deg(p)} p(x/x_0)$
denotes the homogenization of $p$. Define the sets
\be \label{nn}
\boxed{
\baray{rcl}
\widetilde{K}^h & \coloneqq &
\left\{  \tilde{x}\in \r^{n+1}
\left| \baray{l}
\tilde{c}_{i}(\tilde{x})=0~(i \in \mathcal{E}), \\
\tilde{c}_{j}(\tilde{x}) \geq 0~(j \in \mathcal{I}), \,
x_{0} \geq 0
\earay \right. \right\}, \\
\widetilde{K}^c & \coloneqq &  \widetilde{K}^h \cap
      \{  \tilde{x}\in \r^{n+1} :  x_0 > 0  \}, \\
\widetilde{K}   & \coloneqq &   \widetilde{K}^h \cap
      \{ \tilde{x}\in \r^{n+1} : x_0^2 + x^Tx = 1 \}
\earay
}.
\ee

\begin{dfn}[\cite{nie2012discriminants}] \label{closed:inf}
The set $K$ is said to be closed at $\infty$
if $cl(\widetilde{K}^c) = \widetilde{K}^h$.
\end{dfn}

We remark that the above definition of closedness at $\infty$
depends on the describing polynomials for $K$.
Throughout the paper, when the closedness at $\infty$
is mentioned, the describing polynomials
as in \reff{1.1} are clear in the context.
It is interesting to note that the closedness at $\infty$
is a genericity property (see \cite{guo2014minimizing}).
The following conclusion is obvious.

\begin{lemma}[\cite{nie2012discriminants}]\label{lem3.3}
A polynomial $p > 0$ (resp., $p \geq 0$) on $K$ if and only if
$\tilde{p} > 0$ (resp., $\tilde{p} \geq  0$) on $\widetilde{K}^c$.
Moreover, if $K$ is closed at $\infty$, then $p \geq  0$ on $K$
if and only if $\tilde{p} \geq 0$ on $\widetilde{K}$.	
\end{lemma}

Recall that $p^\hm$ denotes the homogeneous part of the highest degree terms of $p$.
Denote
\be \label{set:Ke}
K^e  \coloneqq  \left\{  x \in \mathbb{R}^{n}
\left| \baray{l}
c^\hm_{i}(x)=0~(i \in \mathcal{E}),  \\
c^\hm_{j}(x) \geq 0~(j \in \mathcal{I}), \\
\|x\|^2-1=0
\earay \right. \right\}.
\ee

When $K$ is compact, the interior of $\mathscr{P}_d(K)$
is precisely the set of polynomials in $\re[x]_d$ that are positive on $K$.
When $K$ is unbounded,  the interior of $\mathscr{P}_d(K)$ is more tricky. 
The interiors of $\mathscr{P}_{d}(K)$ and $\mathscr{P}_{\a}(K)$
can be characterized as follows.

\begin{theorem}\label{pro3.7}
Let $K$ be as in \reff{1.1}.
Suppose $K$ is closed at $\infty$.
Then,  we have that
\bit

\item[(i)] For a degree $d>0$, the interior of
$\mathscr{P}_d(K)$ can be given as
\be
int(\mathscr{P}_d(K)) \, =  \,
\left\{p\in \r[x]_d
\left|\baray{l}
p\left|_{ K}\right. >0,\, 
p^{(d)}|_ {K^{e}} >0
\earay \right.
\right \},
\ee
where $p^{(d)}$ denotes homogeneous part of  degree-d terms of $p$. 

\item[(ii)]
For a finite set $\A \subseteq \N^n$, let $d=\deg(\A)$.
If $\r[x]_{\a} \cap int(\mathscr{P}_{d}(K)) \ne \emptyset$, then
\be
int(\mathscr{P}_{\a}(K))  =
 \left\{ p\in \r[x]_{\a}
\left|\baray{l}
p\left|_{ K}\right. >0,\,
p^{(d)}|_ {K^{e}} >0
\earay \right.
\right \}.
\ee

\eit
\end{theorem}
\begin{proof}
	When $K$ is bounded, the conclusions follow directly from item (i) of Theorem \ref{thm3.1} and the fact that $K^e=\emptyset$. We only consider that $K$ is unbounded in the following.
	Note that for the unbounded $K$, the positivity conditions  $p\left|_{ K}\right. >0,\, 
	p^{(d)}|_ {K^{e}}>0$ are equivalent to that $p\left|_{ K}\right. >0$, $\deg(p)=d$, $p^\hm|_ {K^{e}}>0$.

(i) First, consider an arbitrary $p\in int(\mathscr{P}_d(K))$.
Then  $p - \eps \cdot 1 \in  \mathscr{P}_d(K)$ for some $\eps > 0$,
so $p\left|_{ K}\right. >0$.
For every power $\af$ with $|\af| = d$,
we  have $p \pm \eps x^\af \in  \mathscr{P}_d(K)$ for some $\eps > 0$,
which is impossible if $\deg(p) < d$. Hence, $\deg(p) = d$.
Next, we show that $ p^\hm |_ {K^{e}} >0$.
For every point $u \in K^e$, since $\|u\| = 1$,
we must have $u_i \neq 0$ for some $i=1,\dots,n$.
There exists $\epsilon>0$ such that
$p - \epsilon \cdot sign(u_i)^dx_i^d \in \mathscr{P}_d(K)$.
By Lemma~\ref{lem3.3}, we can get
\[
(p -\epsilon \cdot sign(u_i)^d x_i^d)^\hm =
p^\hm-\epsilon \cdot sign(u_i)^d x_i^d \geq 0
\quad \mbox{on} \quad K^e.
\]
Hence, $ p^\hm(u) \geq \epsilon \cdot |u_i|^d > 0$.
Since $u$ is arbitrary in $K^e$, we have $ p^\hm\left|_ {K^{e}} >0 \right.$.

Second, consider $p \in \re[x]_d$ such that
$p\left|_{ K}\right. >0$,  $p^{(d)}\left|_ {K^{e}} >0 \right.$.
Note that $p \in \mathscr{P}_d(K)$.
Suppose otherwise that $p \notin int(\mathscr{P}_d(K))$.
Then there exists a supporting hyperplane $\mc{H}$ for
$\mathscr{P}_d(K)$ through the point $p$.
Let $q \in \re[x]_d$ be the normal direction for $\mc{H}$.
Then $p - \frac{1}{k} q \not\in \mathscr{P}_d(K)$ for all $k \in \N$.
For each $k$, there exists $x^{(k)} \in K$ such that
\begin{equation*}
 p\left( x^{(k)}\right) - \frac{1}{k} q\left(x^{(k)}\right) <0.
\end{equation*}
The sequence $\{x^{(k)}\}_{k=1}^\infty$ must be unbounded.
If otherwise it is bounded, then $\{x^{(k)}\}$ has an accumulation point
and one can assume $x^{(k)} \rightarrow x^* \in K$, without loss of generality.
Hence, we get
\[
p(x^*)  =  \lim\limits_{k \rightarrow \infty} p\left(x^{(k)}\right)-\lim\limits_{k \rightarrow \infty}
\frac{1}{k} q\left(x^{(k)}\right) \leq 0,
\]
which contradicts that $p\left|_{ K}\right. >0$.
Let $y^{(k)}=\frac{x^{(k)}}{\|x^{(k)}\|}$, then the sequence $\{y^{(k)}\}_{k=1}^\infty$ is bounded.
Similarly, we can  assume that $y^{(k)} \rightarrow y^*$.
Clearly, $\|y^*\|=1$. Note that
\[
c^\hm_{i}(y^*) = \lim\limits_{k \rightarrow \infty}
 \frac{c_i(x^{(k)})}{\|x^{(k)}\|^{\deg(c_i)}}
\]
for every $i$. Therefore,
%
%
\[
c^\hm_{i}(y^*) =0 ~(i \in \mathcal{E}), \quad
c^\hm_{j}(y^*)  \geq0 ~(j \in \mathcal{I}) .
\]
This means that $y^* \in K^e$. Furthermore, we have
\begin{equation*}
p^{(d)}(y^*)=\lim\limits_{k\rightarrow \infty}\frac{p(x^{(k)})-
\frac{1}{k} q(x^{(k)})}{ \|x^{(k)}\|^d } \leq 0.
\end{equation*}
It contracts that $p^{(d)}\left|_ {K^{e}} >0 \right.$.
Hence, we have $p \in int(\mathscr{P}_d(K))$.
 	
(ii) Suppose $p\left|_{ K}\right. > 0$, $p^{(d)}|_ {K^{e}} >0$.
By the item (i), we have $p \in int(\mathscr{P}_{d}(K))$,
and hence $p \in int(\pk)$.

On the another hand, suppose  $p \in int(\pk)$.
Let $q$ be a polynomial  in $ \pk \cap int(\mathscr{P}_{d}(K))$.
For $\epsilon>0$ small enough, we have $p-\epsilon q \in \pk$.
Clearly, for  $  x \in K$, $p(x) \geq \epsilon q(x)>0 $.
Moreover, the polynomial $p^\hm-\epsilon q^\hm \geq 0$ on $K^e$,
which implies $p^\hm\left|_ {K^{e}} >0 \right.$.
If $\deg(p) < d $, let $x^{(k)} \in K$ be a sequence such that $\|x^{(k)}\| \rightarrow \infty$.
Denote $y^{(k)}=\frac{x^{(k)}}{\|x^{(k)}\|}$.
Without loss of generality, assume that $y^{(k)} \rightarrow y^*$.
Then we have $y^* \in K^e$ and the following holds
\begin{equation*}
  -\epsilon q^{(d)}(y^*) =
 \lim\limits_{k\rightarrow \infty}\frac{p(x^{(k)}) -
 \epsilon q(x^{(k)})}{\|x^{(k)}\|^{d}} \geq 0,
\end{equation*}
which contradicts that $q^{(d)} \left|_ {K^{e}} >0 \right.$.
Thus, we have $\deg(p) = d$ and $p^{(d)}|_ {K^{e}} >0$.
 \end{proof}

\begin{remark}
%
%
%
It is possible that $\mathscr{P}_d(K)$ has empty interior.
For instance, this is the case for $K=\r^n$ and $d=2k+1$.
For such cases, we should consider the relative interior of $\mathscr{P}_d(K)$.
Clearly, the relative interior of $\mathscr{P}_{2k+1}(\re^n)$
is the interior of $\mathscr{P}_{2k}(\re^n)$.
When $K$ is a general set, it is typically difficult to character the relative interior
of $\mathscr{P}_d(K)$. Moreover, when  $\r[x]_{\a}\cap int(\mathscr{P}_{d}(K)) = \emptyset$,
the conclusion of Theorem~\ref{pro3.7} may not hold.
For instance, when $K= \r^2$
and $\a=\{(2,0),(0,2),(2,2)\}$,
the polynomial $p=x_1^2+x_2^2+x_1^2x_2^2 \in int(\pk)$
but $p^\hm(x)=x_1^2x_2^2$ is not positive  on $K^e$.
\end{remark}

For the power set $\a\subseteq \n^n$,
let $d  = \deg(\A)$. Its homogenization is the set
\[
\tilde{\a}   \coloneqq  \{ \beta \in \n^{n+1}  \mid \beta  =
(d - |\alpha|,\alpha),~\alpha  \in \a \}.
\]

\begin{pro}  \label{lem3.9}

Suppose $K$ is  closed at $\infty$ and $\r[x]_{\a}\cap int(\mathscr{P}_{d}(K)) \ne \emptyset$.
Then, we have $p \in int(\pk)$
if and only if $\hat{p}\left|_{ \widetilde{K}}\right. >0$, where $\hat{p}$ denotes the degree-d homogenization of $p$.
\end{pro}
\begin{proof}
By the assumptions,
$p \in int(\pk)$ is equivalent to that
$ p\left|_{ K}\right. >0, p^{(d)}\left|_ {K^{e}}\right. >0$
(see Theorem~\ref{pro3.7}). This holds if and only if $\hat{p}\left|_{ \widetilde{K}}\right. >0$.
\end{proof}

\subsection{Semidefinite relaxations of $\rk$ and $\pk$}
\label{sec3.2}

Each $\a$-tms $y \in \r^{\a}$ is uniquely determined by its homogenization
$\tilde{y} \in \r^{\tilde{\a}}$ such that
\[
\tilde{y}_{(d-|\alpha|,\alpha)}  \,  =  \,
y_{\alpha} \, (\alpha \in \a),
\]
where $d  = \deg(\A)$. We define the Riesz functional
$\mathscr{L}_{\tilde{y} }$ acting on $\r[\tilde{x}]_{\tilde{\a}}$  as
 \begin{equation*}
 \mathscr{L}_{\tilde{y}}(\sum_{\alpha \in \a}
 p_{\alpha} x_0^{d-|\alpha|} x^{\alpha})=\sum_{\alpha \in \a} p_{\alpha} y_{\alpha}.
\end{equation*}

\begin{thm} \label{thm3.10}
Suppose $K$ is closed at $\infty$ and
$\r[x]_{\a}\cap int(\mathscr{P}_{d}(K)) \ne \emptyset$.
Then, we have $y \in cl(\mathscr{R}_{\a}(K))$ if and only if
$\tilde{y} \in \mathscr{R}_{\tilde{\a}}(\widetilde{K})$.
\end{thm}
\begin{proof}
Suppose $y \in cl(\mathscr{R}_{\a}(K))$, then the Reisz functional
$\l_y$ is $K$-positive (see Theorem \ref{thm3.1}). If $\ell \in \mathscr{P}_{\tilde{\a}} (\widetilde{K})$, then $\ell(1,x) \geq 0$ for  $x \in K$. Hence, we have
\[
\l_{\tilde{y}}(\ell(\tilde{x}))=\l_{y}(\ell(1,x)) \geq 0,
\]
which implies $\l_{\tilde{y}}$ is $\widetilde{K}$-positive.
Note that $\widetilde{K}$ is compact and  there  exists a polynomial  $ \eta\in \r[\tilde{x}]_{\tilde{\a}}$ such that $\eta\left|_{ \widetilde{K}}\right. >0$
(see Proposition \ref{lem3.9}),  thus conditions of Theorem \ref{thm2.1}
are satisfied and $\tilde{y} \in \mathscr{R}_{\tilde{\a}}(\widetilde{K})$.
On the another hand, if $\tilde{y} \in \mathscr{R}_{\tilde{\a}}(\widetilde{K})$,
then $\l_{\tilde{y}}$ is $\widetilde{K}$-positive.
 For  $p\in  \pk$, the following holds
\[
\sum_{\alpha \in \a} p_{\alpha} x_0^{d-|\alpha|} x^{\alpha}=x_0^{d-\deg(p)}
\tilde{p} \geq 0,~ \forall \tilde{x} \in \widetilde{K}.
\]
Hence, we have
\[
\l_y(p) = \mathscr{L}_{\tilde{y}}
(\sum_{\alpha \in \a} p_{\alpha} x_0^{d-|\alpha|} x^{\alpha}) \geq 0.
\]
By Theorem \ref{thm3.1}, it implies that $y \in cl(\mathscr{R}_{\a}(K))$.
\end{proof}

\begin{remark}\label{approa}
It is possible that $y \in cl(\mathscr{R}_{\a}(K))$
while $y$ admits no $K$-representing measures. In such case,
 $y$ is the limit of a sequence of
$\A$-truncated multi-sequences that admit finitely atomic $K$-representing measures.
When $y \in cl(\mathscr{R}_{\a}(K))$, we have
$\tilde{y} \in \mathscr{R}_{\tilde{\a}}(\widetilde{K})$.
Suppose the measure
$\nu = \tilde{\lambda}_{1} \delta_{(\tau_1, v_1)} + \cdots +
	\tilde{\lambda}_{r} \delta_{(\tau_r, v_r)}$
is a finitely atomic $\widetilde{K}$-representing measure of $\tilde{y}$, where $(\tau_i, v_i) \in \widetilde{K}$, $\tilde{\lambda}_{i}>0$. If $y \notin \mathscr{R}_{\a}(K)$,  then there exist $1\leq i_1,\dots,i_t\leq r$ such that  $\tau_{i_1}=\cdots=\tau_{i_t}=0$.
When $K$ is closed at $\infty$, we have $\tilde{u}^{(i_1,k)}\rightarrow(0, v_{i_1}),\dots,\tilde{u}^{(i_t,k)}\rightarrow (0, v_{i_t})$  for
\[
\tilde{u}^{(i_1,k)}=({u}^{(i_1,k)}_0,{u}^{(i_1,k)}),\dots,
\tilde{u}^{(i_t,k)}=({u}^{(i_t,k)}_0,{u}^{(i_t,k)}) \in \widetilde{K}^c .
\]
Let $\mu^{(k)} \coloneqq \lmd_{1}^{(k)} \delta_{u_1^{(k)}} +
\cdots + \lmd_{r}^{(k)} \delta_{u_r^{(k)}}$, where
\[
\lmd_i^{ (k)}  = \left\{\baray{l}
\tilde{\lambda}_{i}\tau_i^d~\text{ if } i\neq i_1,\dots,i_t,\\
\tilde{\lambda}_{i}({u}^{(i,k)}_0)^d~\text{ if } i= i_1,\dots,i_t,\\
\earay
\right.
\]
\[
u_i^{(k)}=\left\{\baray{l}
\frac{\nu_i}{\tau_i}~\text{ if } i\neq i_1,\dots,i_t,\\
\frac{{u}^{(i,k)}}{{u}^{(i,k)}_0}~\text{ if } i= i_1,\dots,i_t.\\
\earay
\right.
\]
Let $z^{(k)}=\int_K [x]_{\a} \mathrm{~d}\mu^{(k)}$ for each $k$,
then $z^{(k)}\in \mathscr{R}_{\a}(K)$ and
$z^{(k)}\rightarrow y$ as $k \to \infty$.
\end{remark}

For the semialgebraic set $K$ as in \reff{1.1},
denote the constraining polynomial tuples
\be \label{cert}
\boxed{
\baray{rcl}
\tilde{c}_{eq} & \coloneqq &  \left(\tilde{c}_{i}(\tilde{x})\right)_{i \in \mathcal{E}} \cup\left\{\|\tilde{x}\|^{2}-1\right\}, \\
\tilde{c}_{in} & \coloneqq &  \left(\tilde{c}_{j}(\tilde{x})\right)_{j \in \mathcal{I}} \cup\left\{x_{0}\right\}
\earay
}.
\ee
In the following, we give a certificate about nonexistence of $K$-representing measures.

\begin{lemma} \label{infeasible}
Suppose $K$ is closed at $\infty$ and
$\r[x]_{\a}\cap int(\mathscr{P}_{d}(K)) \ne \emptyset$.
Then $y \notin cl(\mathscr{R}_{\a}(K))$ if and only if there exists
$q\in \r[\tilde{x}]_{\tilde{\a}}$ such that
\begin{equation}  	\label{pro3.11}
 \l_{\tilde{y}}(q) <0,  \quad
q \in \ideal{\tilde{c}_{eq}}+\qmod{\tilde{c}_{in}}.
\end{equation}
\end{lemma}

\begin{proof}
If $y \notin cl(\mathscr{R}_{\a}(K))$, then $\tilde{y} \notin
\mathscr{R}_{\tilde{\a}}(\widetilde{K})$ and there exists a polynomial
$q_1 \in \mathscr{P}_{\tilde{\a}} (\widetilde{K})$ such that
$\l_{\tilde{y}}(q_1) <0$ (cf. Theorem \ref{thm3.10}). 	
By Proposition \ref{lem3.9}, there  exists a polynomial
$\eta \in \r[\tilde{x}]_{\tilde{\a}}$  such that
$\eta\left|_{ \widetilde{K}}\right. >0$.
%
%
For  $\epsilon >0$ small enough,  we have
 $$
\left. q_1+\epsilon \eta \right|_{\widetilde{K}} >0,~
\l_{\tilde{y}}(q_1+\epsilon \eta) <0.
 $$
Note that  $\widetilde{K}$ is archimedean. By Theorem \ref{thm2.2}, the following holds
\[
q_1+\epsilon \eta \in \ideal{\tilde{c}_{eq}}+\qmod{\tilde{c}_{in}}.
\]
Hence, the polynomial $q \coloneqq q_1+\epsilon \eta$ satisfies \reff{pro3.11}.
 For the contrary, suppose otherwise $y \in cl(\mathscr{R}_{n,d}(K))$.
 Clearly, if $q \in  \ideal{\tilde{c}_{eq}}+\qmod{\tilde{c}_{in}}$,
 then  $q \geq 0$ on $\widetilde{K}$. It implies that
 $\l_{\tilde{y}}(q) \geq 0$, which is a contradiction.
\end{proof}

In the following, we give convergent semidefinite relaxations for $\rk$ and $\pk$.
For each $k \in \n$,
we consider the following $k$th order relaxation for the cone $\mathscr{R}_{\a}(K)$
\be  \label{cone:Rk(K)}
\mathscr{R}^{(k)}(K)    \coloneqq
\left\{  y\in \r^\a
\left|\baray{l}
\tilde{y}= w|_{\tilde{\a}},\,
w \in \mathbb{R}^{\mathbb{N}_{2 k}^{n+1}} ,  \\
L_{\tilde{c}_{i}}^{(k)}[w] = 0~ (i\in \mathcal{E}), \\
L_{\|\tilde{x}\|^2-1 }^{(k)}[w] = 0, \\
L_{\tilde{c}_{j}}^{(k)}[w] \succeq 0~(j\in \mathcal{I}), \\
L_{x_0}^{(k)}[w] \succeq 0, \, M_k[w] \succeq 0
\earay \right.
\right\}.
\ee
Recall that the homogenization of a polynomial $p(x)$ is
$\tilde{p} \coloneqq  x_0^{\deg(p) } p(x/x_0)$.
The $k$th order SOS approximation for the cone $\mathscr{P}_{\a}(K)$ is
\begin{equation}
\mathscr{P}^{(k)}(K) \, \coloneqq \,
\Big\{ p \in \re[x]_{\A}: \,  \tilde{p} \in
\ideal{\tilde{c}_{eq}}_{2k}+\qmod{\tilde{c}_{in}}_{2k}
\Big \}.
\end{equation}
%
%

\begin{thm}
Suppose  $K$ is closed at $\infty$ and
$\r[x]_{\a}\cap int(\mathscr{P}_{d}(K)) \ne \emptyset$.
Then, we have that
\be \label{cvg:Ra(K)}
\bigcap\limits_{k=1}^{\infty} \mathscr{R}^{(k)}(K) = cl(\mathscr{R}_{\a}(K)),
\ee
\be \label{cvg:Pa(K)}
int\left(\mathscr{P}_{\a}(K)\right) \subseteq \bigcup_{k=1}^{\infty}
\mathscr{P}^{(k)}(K) \subseteq \mathscr{P}_{\mathcal{A}}(K).
\ee
\end{thm}

\begin{proof}
In \reff{cvg:Ra(K)}, the containment ``$\supseteq$" follows from
\[
cl(\mathscr{R}_{\a}(K))=\mathscr{R}_{\tilde{\a}}(\widetilde{K})
\subseteq \bigcap\limits_{k=1}^{\infty} \mathscr{R}^{(k)}(K).
\]
To prove the reverse containment ``$\subseteq$",
it is enough to show that if $y \notin cl(\mathscr{R}_{\a}(K))$,
then $y \notin \mathscr{R}^{(k)}(K)$ for some $k$ big enough.
Suppose $y \notin cl(\mathscr{R}_{\a}(K))$, then there exists
$q \in \r[\tilde{x}]_{\tilde{\a}}$ satisfying
\begin{equation*}
\langle q, \tilde{y}\rangle <0, \, q \in
\ideal{\tilde{c}_{eq}}_{2k_1}+\qmod{\tilde{c}_{in}}_{2k_1},
\end{equation*}
for some integer $k_1$ (see Lemma \ref{infeasible}).
If $y \in \mathscr{R}^{(k_1)}(K)$, then
$\tilde{y}=\left.w\right|_{\tilde{\mathcal{A}}}$ for
some $w$ satisfying \reff{cone:Rk(K)}.
Since $w$ belongs to the dual cone of $\ideal{\tilde{c}_{eq}}_{2k_1}+\qmod{\tilde{c}_{in}}_{2k_1}$,
we get the contradiction
\[
0 > \langle q, \tilde{y}\rangle = \langle q, w\rangle \geq 0 .
\]
Thus, $y \notin \mathscr{R}^{(k)}(K)$ for all $k\geq k_1$.
So, the relation \reff{cvg:Ra(K)} holds.

Now, we prove \reff{cvg:Pa(K)}. If $p \in \mathscr{P}^{(k)}(K)$,
we have $\tilde{p} \geq 0$ on $\widetilde{K}$.
Then for $x\in K$, $p(x)=\tilde{p}(1,x) \geq 0$, which implies
$p \in\mathscr{P}_{\mathcal{A}}(K) $.
If $p \in int\left(\mathscr{P}_{\a}(K)\right)$,
then $\tilde{p}\left|_{ \widetilde{K}}\right. >0$.
By Theorem \ref{thm2.2},
we have $\tilde{p} \in \ideal{\tilde{c}_{eq}}+\qmod{\tilde{c}_{in}}$, i.e., $p\in
 \bigcup_{k=1}^{\infty} \mathscr{P}^{(k)}(K)$.
\end{proof}

\section{Solving GTMPs}
\label{Generalizedtkmp}

In this section, we give a Moment-SOS approach for
solving generalized truncated moment problems
with unbounded sets, based on homogenization.
Let $K$ be the semialgebraic set as in \reff{1.1}.
Let $\a \subseteq \n^n$ be a finite power set.
We have seen that the truncated moment system \reff{bi=<aiy>}
is equivalent to the existence of an $\a$-tms $y$ satisfying
\be \label{eq5.1}
\boxed{
\baray{rcl}
\langle a_i,y\rangle  & = & b_i \, \, (1 \le i \le m_1), \\
\langle a_i,y\rangle  & \geq &  b_i \, \, ( m_1 < i \le m), \\
y &\in & \mathscr{R}_{\a}(K)
\earay
}
\ee
for some given polynomials $a_1, \ldots, a_m \in \re[x]_{\a}$.

When $K$ is unbounded, a difficulty for solving \reff{eq5.1} is that the cone
$\mathscr{R}_{\a}(K)$ is typically not closed.
In computational practice, it is usually more convenient for
considering its closure moment system
\be 	\label{eqq5.6}
\boxed{
\baray{rcl}
\langle a_i,y\rangle  & = & b_i \, \, (1 \le i \le m_1),  \\
\langle a_i,y\rangle  & \geq &  b_i\, \, ( m_1 < i \le m),  \\
y &\in & cl(\mathscr{R}_{\a}(K)) 
\earay
}.
\ee
It is important to observe the following:
if the moment system $(\ref{eq5.1})$ is feasible,
then the closure system \reff{eqq5.6} must also be feasible;
if \reff{eqq5.6} is infeasible, then \reff{eq5.1} must also be infeasible.
There may exist boundary cases that
\reff{eq5.1} is infeasible while \reff{eqq5.6} is feasible,
which is a difficult case in computational practice.

Let $d=\deg(\a)$. Recall that each $y \in \re^{\A}$ can be viewed as a
homogeneous tms $\tilde{y}  \in \re^{\tilde{\a}}$ labelled such that
\[
\tilde{y}_{(d -|\alpha|,\alpha)} \, = \, y_{\alpha}
\quad \mbox{for every} \quad \alpha \in \A.
\]
%
%
For each $i=1,\dots,m$, let
\[
\hat{a}_i  \, \coloneqq \,  x_0^{d}a_i( x/x_0 )
\]
be the degree-$d$ homogenization of $a_i$.
Note that $d > \deg(a_i)$ is possible.
For a degree $d_1 > d$, select a generic polynomial $f \in \re[\tilde{x}]_{d_1}$
and consider the following linear conic optimization problem
\be  \label{eq5.2}
\left\{ \baray{rl}
\min & \langle f, w\rangle  \\
\st & \langle \hat{a}_i,w\rangle   =  b_i \, \, ( 1 \le i \le m_1),\\
&\langle \hat{a}_i,w\rangle   \geq  b_i \, \, ( m_1 < i \le m), \\
 &  w \in \mathscr{R}_{d_1}(\widetilde{K}).
\earay \right.
\ee
The dual optimization problem of \reff{eq5.1} is
\be  \label{eq5.3}
\left\{ \baray{rl}
\max & \sum\limits_{i=1}^m b_i\theta_i  \\
\st &  f - \sum\limits_{i=1}^m \theta_i\hat{a}_i
    \in \mathscr{P}_{d_1}(\widetilde{K}),\\
    & \theta_i \geq 0 \,\, (m_1 < i \le m).
\earay \right.
\ee
For convenience, denote the set
\be \label{set:H0}
H_0 = \mbox{Span} \{a_1, \ldots, a_{m_1} \} -
\mbox{Cone} \{a_{m_1+1}, \ldots, a_{m} \}.
\ee
The following are some properties of the pair \reff{eq5.2}-\reff{eq5.3}.

\begin{pro} \label{pro1}
Let $d_1 > d $ be even.	Suppose the system \reff{eq5.1} is feasible.
\bit

\item[(i)]
If  $f \in int(\Sigma[\tilde{x}]_{d_1})$, then  strong duality holds
and the minimum value of \reff{eq5.2} is achievable.

\item[(ii)]
Suppose $f$ is generic in $\Sigma[\tilde{x}]_{d_1}$,
then the optimization problem $(\ref{eq5.2})$ has a unique minimizer $w^*$.
Moreover, every $\widetilde{K}$-representing measure $\mu$ of $w^*$
is $r$-atomic with $r \leq m$.

\eit
\end{pro}

\begin{proof}
(i) Since $\widetilde{K}$ is compact, the cone $\mathscr{R}_{d_1}(\widetilde{K})$
is closed (see Theorem \ref{thm3.1}). The feasible set of \reff{eq5.2} is nonempty,
since \reff{eq5.1} is feasible.
When $f \in int(\Sigma[\tilde{x}]_{d_1})$,  the dual optimization \reff{eq5.3}
has an interior point.
Hence the strong duality holds and
the optimization \reff{eq5.2} must achieve its minimum value (see \cite{Bert}).

(ii) If $f$ is generic in $\Sigma[\tilde{x}]_{d_1}$, the objective
$\langle f, w\rangle$ is a generic linear functional. Note that
$\widetilde{K}$ is compact, the uniqueness can be  implied
by Proposition~5.2 of \cite{2014The}. Since the minimizer $w^*$ is unique,
it must be an  extreme  point of the convex set
\be \label{extre}
E = \left\{  w \in \mathscr{R}_{d_1}(\widetilde{K})
\left|\baray{l}
\langle \hat{a}_i,w\rangle   =  b_i \, ( 1 \le i \le m_1), \\
\langle \hat{a}_i,w\rangle   \geq  b_i \, (m_1 < i \le m)
\earay \right.
\right\}.
\ee
Suppose $\mu$ is a $\widetilde{K}$-representing measure  of $w^*$.
Let $J \subseteq \{ m_1+1,\ldots, m\}$
be the index set of active inequality constraints for $w^*$.
In the following, we prove that $\mu$ is $r$-atomic with $r\leq m$.

First, we consider the case that $\mu$ is finitely atomic,
say, $\mu=\lambda_{1}\delta_{u_1}+\cdots+\lambda_{r}\delta_{u_r}$,
where $u_i\in\r^{n+1}$ are distinct and $\lambda_{i}>0$.
We need to show that $r\leq m$. Suppose otherwise that $r> m$.
We claim that there exist $\lambda_{1}^*,\dots,\lambda_{r}^*\in \r$
and $p^*\in \r[\tilde{x}]_{d_1} \setminus \mbox{Span}\{\hat{a}_1,\dots,\hat{a}_m\}$
such that
\be   \label{atomee}
\boxed{
\baray{rcl}
\sum\limits_{k=1}^r \lambda^*_k \hat{a}_i(u_k) &=& 0 \,\,
 (i \in [m_1] \cup J),  \\
\sum\limits_{k=1}^r \lambda^*_k p^*(u_k) &\neq&   0
\earay
}.
\ee
If such $p^*$ and $\lmd_k^*$ do not exist, then for every $
p \in \r[\tilde{x}]_{d_1} \setminus \mbox{Span}\{\hat{a}_1,\dots,\hat{a}_m\}$,
the vector $(p(u_1),\dots,p(u_r))$ is a linear combination of the vectors
\[
(\hat{a}_i(u_1),\dots,\hat{a}_i(u_r)), i \in [m_1] \cup J .
\]
We can extend $\{ \hat{a}_i: i\in [m_1] \cup J \}$
to a basis of $\r[\tilde{x}]_{d_1}$ by choosing
\[
p_1,\dots,p_t \in \r[\tilde{x}]_{d_1}
\setminus \mbox{Span}\{\hat{a}_1,\dots,\hat{a}_m\}.
\]
For $j=1,\dots,t$ and $k=1,\dots,r$,
there exist scalars $s_{j,i}$ such that
\[
\sum\limits_{i \in [m_1] \cup J} s_{j,i} \hat{a}_i(u_k)= p_j(u_k).
\]
Let $q_j \coloneqq  p_j- \sum\limits_{i \in [m_1] \cup J} s_{j,i} \hat{a}_i$
for each $j=1,\dots,t$. Let $I^*$ be the ideal generated by $q_1,\dots,q_t$.
We show that the vectors $\left[\hat{a}_i\right]$,
with $i\in [m_1] \cup J$,
span the quotient space $\r[\tilde{x}] /I^*$.
For each $p\in \r[\tilde{x}]_{d_1}$, we write that
\[
p = \sum\limits_{i \in [m_1] \cup J} c_i\hat{a}_i+\sum\limits_{j=1}^t d_jp_j ,
\]
for some scalars $c_i, d_i \in \r$. Note that
\[
 p \equiv  \sum\limits_{i \in [m_1] \cup J} c_i\hat{a}_i+\sum\limits_{j=1}^t d_jp_j
 \equiv  \sum\limits_{i \in [m_1] \cup J} c_i\hat{a}_i+
 \sum\limits_{j=1}^t d_j
 \Big(\sum\limits_{i \in [m_1] \cup J}  s_{j,i} \hat{a}_i \Big)
 \quad \bmod \quad I^*.
\]
Thus, the equivalent class $[p]$ of $p$ can be spanned by vectors
$[\hat{a}_i]$ $(i\in [m_1] \cup J)$. For $\left|\beta\right|=d_1+1$, we have $x^{\beta}=x_jx^{\alpha}$ for some $j\in [n]$
and $\left|\alpha\right|=d_1$.
Since $d_1 >d$ and $\deg(x_j\hat{a}_i)\leq d_1$, it follows that
$[x^{\beta}]$ can be spanned by $[\hat{a}_i]$ $(i\in [m_1] \cup J)$.
By induction,  we know that the vectors
$[\hat{a}_i]$ $(i\in [m_1] \cup J)$
span $\r[\tilde{x}]/I^*$.
By Theorem 2.6 of \cite{laurent2009sums}, we have
$\operatorname{dim} \r[\tilde{x}]/ I^* \leq m_1+|J|$,
and the number of common zeros of polynomials $q_j$ is no more than $m_1+|J|$.
However, for $k=1,\dots,r$, we have
\[
q_j(u_k) \, = \, p_j(u_k)-
\sum\limits_{i \in [m_1] \cup J} s_{j,i}(u_k) \hat{a}_i(u_k)=0,
\]
which contradicts with $r>m\geq m_1+|J|$.
Hence, the equation \reff{atomee} holds.
Let $\epsilon>0$ be small enough such that
$\lambda_{k}\pm\epsilon\lambda_{k}^*>0$ $(k=1,\dots,r)$, and denote
\[
 \mu_1=(\lambda_{1}+\epsilon\lambda_{1}^*)\delta_{u_1}+\cdots+(\lambda_{r}+
 \epsilon\lambda_{r}^*)\delta_{u_r},~
\]
\[
\mu_2=(\lambda_{1}-\epsilon\lambda_{1}^*)\delta_{u_1}+\cdots+(\lambda_{r}-
\epsilon\lambda_{r}^*)\delta_{u_r}.~
\]
Let
\[
w_{1}=\int_{\widetilde{K}}[\tilde{x}]_{d_1} \mathrm{~d} \mu_{1}, \quad w_{2}=\int_{\widetilde{K}}[\tilde{x}]_{d_1} \mathrm{~d} \mu_{2}.
\]
Note that for $j\in \{m_1+1,\dots,m\}\backslash J$, we have
$ \langle \hat{a}_j,w\rangle  >  b_j$. Thus if $\epsilon$ is small enough,  we can easily check $w_1,w_2 \in E$,  $w^*=\frac{1}{2}\left(w_{1}+w_{2}\right)$, and  $w_1\neq w^*$, $w_2\neq w^*$,
which contradicts that $z$ is an extreme point of $ E$.

Second, consider the case that $\mu$ is not finitely atomic.
It can be reduced to the case of finitely atomic measures.
We refer to \cite[Lemma 3.5]{2014The}. Thus, every $\widetilde{K}$-representing measure
$\mu$ of $w^*$ must be $r$-atomic with $r \leq m$.
\end{proof}

For an order $k \geq \lceil d_1/2 \rceil$, the $k$th order moment relaxation of \reff{eq5.2} is
\be  \label{eq5.4}
\left\{ \baray{cl}
\min &  \langle f, w \rangle  \\
\st  &  \langle \hat{a}_i,w\rangle  =  b_i \, \, ( 1 \le  i \le m_1),  \\
&\langle \hat{a}_i,w\rangle   \geq  b_i \, \, (m_1 < i \le m),\\
&L_{\tilde{c}_{i}}^{(k)}[w] = 0~ (i\in \mathcal{E}), \\
& L_{\|\tilde{x}\|^2-1 }^{(k)}[w] = 0, \\
&L_{\tilde{c}_{j}}^{(k)}[w] \succeq 0~(j\in \mathcal{I}), \\
&  L_{x_0}^{(k)}[w] \succeq 0,  \\
& M_k[w] \succeq 0, \, w \in \mathbb{R}^{\mathbb{N}_{2 k}^{n+1}} .
\earay \right.
\ee
The dual optimization problem of $(\ref{eq5.4})$ is
\be \label{eq5.5}
\left\{ \baray{cl}
\max &  b_1 \theta_1 + \cdots + b_m \theta_m  \\
\st & f - \sum\limits_{i=1}^m \theta_i\hat{a}_i
\in  \ideal{\tilde{c}_{eq}}_{2k}+\qmod{\tilde{c}_{in}}_{2k},\\
& \theta_{m_1+1} \geq 0,\ldots, \theta_{m} \geq 0 .
\earay \right.
\end{equation}
In the above, $\tilde{c}_{eq}$ and $\tilde{c}_{in}$
are from \reff{cert}.
The following are some properties of the relaxations
\reff{eq5.4}-\reff{eq5.5}.

\begin{thm}
 Let $d_1 > d $ be even. Then we have:

\bit
	
\item[(i)] If the moment relaxation \reff{eq5.4}
is infeasible for some order $k$,
then the truncated moment system \reff{eq5.1} is  infeasible.
	
\item[(ii)] Suppose $K$ is closed at $\infty$ and
$H_0\cap int(\mathscr{P}_{d}(K)) \ne \emptyset$.
Then, the closure moment system  $(\ref{eqq5.6})$
is infeasible if and only if
the moment relaxation $\reff{eq5.4}$ is infeasible for some order $k$.
	
\item[(iii)]
Suppose  $K$ is closed at $\infty$ and $w^{*}$ is a minimizer of $(\ref{eq5.4})$ such that the truncation
$\left.w^{*}\right|_{2 t}$  $(2 t \geq d)$ is flat.
Then, the closure moment system \reff{eqq5.6} is feasible.
Moreover, if $\nu$ is the $\widetilde{K}$-representing measure
for $\left.w^{*}\right|_{2 t}$ and
\be \label{suppmu:x0>0}
\supp{\nu} \subseteq \{   x_0 > 0\},
\ee
then the moment system $(\ref{eq5.1})$ is also feasible.

\eit
\end{thm}

\begin{proof}
(i) This conclusion is obvious, because
\reff{eq5.4} is a relaxation of \reff{eq5.2}.
If  \reff{eq5.1} is feasible,
then \reff{eq5.2} and \reff{eq5.4} are also feasible.

(ii)
The ``if" part follows from Theorem \ref{thm3.10}, so we only need to show the ``only if" part.
Suppose the system \reff{eqq5.6} is infeasible, i.e.,
the following optimization
\be  \label{eq5.7}
\left\{ \baray{cl}
\max  & 0  \\
\st & \langle \hat{a}_i,w\rangle   =  b_i \, \, ( 1 \le  i \le m_1), \\
&\langle \hat{a}_i,w\rangle   \geq  b_i \, \, ( m_1 < i \le m),\\
&  w \in \mathscr{R}_{\a}(\widetilde{K})
\earay \right.
\ee
is infeasible. The dual problem of $(\ref{eq5.7})$ is
\be \label{eq5.8}
\left\{ \baray{rl}
\min &  \sum\limits_{ i \in [m_1] } b_i t_i -
           \sum\limits_{m_1 < j \leq m}   b_j t_j  \\
\st &  \sum\limits_{ i \in [m_1] }  t_i\hat{a}_i-
        \sum\limits_{m_1 < j \leq m}   t_j\hat{a}_j
   \in \mathscr{P}_{\a}(\widetilde{K}),  \\
   &  t_{m_1+1} \geq 0, \ldots, t_m \ge 0 .
\earay \right.
\ee
Since $H_0\cap int(\mathscr{P}_{d}(K)) \ne \emptyset$,
we know \reff{eq5.8} has an interior point.
By the strong duality, the optimization \reff{eq5.8}
must be unbounded below, i.e.,
there exists  $\lmd_1^*,\cdots,\lmd_{m_1}^*$, $\lmd_{m_1+1}^*\geq0, \ldots, \lmd_m^*\geq0$ such that
\begin{align*}
 b(\lmd^*) \coloneqq &
  \sum\limits_{ i \in [m_1] } b_i \lmd_i^*-
  \sum\limits_{m_1 < j \leq m}  b_j \lmd_j^* <0, \\
 \hat{a}(\lmd^*) \coloneqq &
  \sum\limits_{ i \in [m_1] } \lmd_i^* \hat{a}_i-
  \sum\limits_{m_1 < j \leq m} \lmd_j^* \hat{a}_j
 \in \mathscr{P}_{\a}(\widetilde{K}).
\end{align*}
Since $H_0\cap int(\mathscr{P}_{d}(K)) \ne \emptyset$,
let $\bar{\lmd}$ be such that $\hat{a}(\bar{\lmd}) > 0$ on $\widetilde{K}$.
Then for $\eps > 0$ small enough, $\hat{a}(\eps \bar{\lmd} + \lmd^* ) > 0$ on $\widetilde{K}$.
By Putinar's Positivstellensatz (see Theorem~\ref{thm2.2}), the following holds
\[
\hat{a}(\eps \bar{\lmd} + \lmd^* )   \in
\ideal{\tilde{c}_{eq}}_{2k_1}+\qmod{\tilde{c}_{in}}_{2k_1},
\quad
b(\eps \bar{\lmd} + \lmd^* )<0,
\]
for some $k_1$ and $\epsilon>0$ small enough.
This means that \reff{eq5.5} has an improving ray,
so \reff{eq5.4} must be infeasible.

(iii) Since the truncation $\left.w^{*}\right|_{2 t}$  is flat,
it admits a unique $\widetilde{K}$-representing measure $\nu$.
Moreover, the measure $\nu$ is finitely atomic (see Theorem~\ref{thm1.1}), say,
\[
\nu = \tilde{\lambda}_{1} \delta_{(\tau_1, v_1)} + \cdots +
\tilde{\lambda}_{r} \delta_{(\tau_r, v_r)},
\]
with each point $(\tau_i, v_i) \in \widetilde{K}$.
Hence, the dehomogenized tms of $\left.w^{*}\right|_{\tilde{\a}}$ is a feasible solution of the closure moment system \reff{eqq5.6}.
Moreover, when the assumption \reff{suppmu:x0>0} holds,
we have $\tau_i > 0$ for every $i$.
Let $u_i = v_i/\tau_i$, then $u_i \in K$. So the measure
\[
\mu =  \lmd_{1} \delta_{u_1} + \cdots + \lmd_{r} \delta_{u_r},
\]
where $\lmd_i = \tilde{\lambda}_{1} \tau_i^d$, is a $K$-representing measure
for the dehomogenized tms of $\left.w^{*}\right|_{2 t}$.
This implies that the system $(\ref{eq5.1})$ is feasible.
\end{proof}

In view of the above properties,
we get the following algorithm.

\begin{algorithm}\label{alg2}
(A Moment-SOS algorithm for solving the GTMP)
For the semialgebraic set $K$ as in \reff{1.1},
let $\widetilde{K}$ be as in $(\ref{nn})$ and let
$d_K$ be as in \reff{deg:dK}.

\bit

\item [Input:]  The polynomials $a_1, \ldots, a_{m_1}, \ldots, a_m$
and constants $b_1, \ldots, b_{m_1}, \ldots, b_m$,
and a finite power set $\a \subseteq \n^n$.

\item [Step~0]
Let $d = \deg(\A)$ and let $\hat{a}_i = x_0^d a_i(x/x_0)$.
Choose an even integer $d_1>d$  and a generic $f  \in \Sigma[\tilde{x}]_{d_1}$.
Let $k \coloneqq \frac{d_1}{2}$.

\item [Step~2]  Solve the Moment-SOS relaxation $\reff{eq5.4}-\reff{eq5.5}$.
If $\reff{eq5.4}$ is infeasible, output that
\reff{eq5.1} is infeasible.
If $\reff{eq5.4}$ is feasible, solve it for a minimizer $w^{*}$.
Let $t = \max \{d_K, \lceil d/2 \rceil \}$.

\item [Step~3]
Check whether the rank condition 
\be \label{flat:condi}
\rank M_{t-d_K}[w^{*}|_{2t}] \, = \, \rank M_{t}[w^{ *}|_{2t}].
\ee
is satisfied or not. If yes, compute the finitely atomic measure $\nu$ for $w^{ *}|_{2t}$:
\[
\nu= \rho_1 \delta_{\tilde{u}_{1}}+\cdots+ \rho_r \delta_{\tilde{u}_{r}},
\]
where $r=\operatorname{rank} M_{t}\left(\left.w^{*}\right|_{2 t}\right),$
$\tilde{u}_i = (\tau_i, v_i) \in \widetilde{K}$  and $\rho_i>0$.
If all $\tau_i > 0$,  let
$\mu  \coloneqq  \lambda_{1} \delta_{u_1}+\cdots+\lambda_{r} \delta_{u_r}$,
where each $\lambda_i = \rho_i \tau_i^{d}$, $u_i=v_i/\tau_i$,  and stop. Otherwise, go to Step 4.

\item [Step~4]  If $t < k$, let $t \coloneqq t+1$ and go to Step~3.
Otherwise, let $k \coloneqq k+1$ and go to Step 2. 

\item [Output]
The finitely atomic $K$-representing measure $\mu$ for $y$ satisfying \reff{eq5.1},
or a detection that \reff{eq5.1} is infeasible.

\eit

\end{algorithm}

\begin{remark}
There are the cases that $w^{ *}|_{2t}$ is flat while its representing measure
$\nu$ has an  atom $\tilde{u}_{l}=(\tau_l, v_l)$
with the $x_0$-coordinate $\tau_l = 0$. For such a case,
we cannot obtain
a $K$-representing measure satisfying \reff{eq5.1},
although \reff{eqq5.6} is feasible.
When this happens, it is possible that \reff{eq5.2}
has other feasible points that have a finitely atomic
$\widetilde{K}$-representing measure
whose support has positive $x_0$-coordinates.
There are two ways to get a feasible point for \reff{eq5.1}:
\bit
\item[(i)]
Choose a different generic $f$ to get a new feasible point of \reff{eq5.2}.

\item[(ii)] Strength the constraint $x_0\geq0$ in $\widetilde{K}$
to $x_0\geq \epsilon$ for some $\epsilon >0$ small.

\eit
However, no matter if  a finitely atomic $\widetilde{K}$-representing measure
whose support has positive $x_0$-coordinates can be obtained,
we can always approximately get a feasible point of \reff{eq5.1}
by a sequence of finitely atomic $K$-representing measures.
We refer to Remark~\ref{approa} for how to do this.
\end{remark}

The following is the convergence result for Algorithm~\ref{alg2}.

\begin{thm}
Let $d_1 > d $ be even. Suppose the moment system \reff{eq5.1} is feasible and $f$
is generic in $\Sigma[\tilde{x}]_{d_1}$.
Then, the moment relaxation \reff{eq5.4} has an optimizer $w^k$
for every order $k$.
Moreover, we also have:
\bit

\item[(i)] For all t big enough, the sequence $\left\{w^{k}|_{2 t}\right\}$
is bounded and all its accumulation points are flat.
Moreover, each of its accumulation points
admits a $r$-atomic $K$-representing measure with $r \leq m$.
	
\item[(ii)]
Assume that  $(\ref{eq5.3})$ has a maximizer $\theta^{*}$ such that
\be \label{finite}
f - \sum\limits_{i = 1}^m \theta_i^*\hat{a}_i  \in
\ideal{\tilde{c}_{eq}}_{2k_1}+\qmod{\tilde{c}_{in}}_{2k_1}
\ee
and the optimization problem
\be \label{lo:finite}
\left\{ \baray{rl}
\min &  f(\tilde{x})- \sum\limits_{i = 1}^m  \theta_i^*\hat{a}_i(\tilde{x})  \\
\st &   \tilde{x} \in \widetilde{K}
\earay \right.
\ee
has finitely many KKT points. Then, there exists $t>0$ such that
the truncation $\left.w^{ k}\right|_{2 t}$ is flat for all $k$ big enough.

\eit
\end{thm}
\begin{proof}
	
 Since $f$ is generic in $\Sigma[\tilde{x}]_{d_1}$, we can assume
$f \in int(\Sigma[\tilde{x}]_{d_1})$. Hence, we know   $(\ref{eq5.5})$ has an interior point,
so the optimal values of $(\ref{eq5.4})$ and $(\ref{eq5.5})$ are equal,
and $(\ref{eq5.4})$ has an optimal minimizer $w^k$,
by the strong duality theorem (see \cite{Bert}).

(i) First, we show the  sequence $\left\{w^{k}|_{2 t}\right\}_{k=t}^\infty$
is bounded. Note that $f \in int(\Sigma[\tilde{x}]_{d_1})$ and $M_k[w^k] \succeq 0$. We have that $f-\epsilon \in\Sigma[\tilde{x}]_{d_1}$ for $\epsilon>0$ small enough and $\langle f-\epsilon,w^k\rangle\geq 0$. Let $OPT(f)$ be the optimal value of \reff{eq5.2}.  Since $(\ref{eq5.4})$ is a relaxation of \reff{eq5.2}, we have
\be \label{w:bound}
(w^k)_0\leq \langle f,w^k\rangle/\epsilon \leq OPT(f)/\epsilon.
\ee
Hence, the sequence $\{(w^k)_0\}$ is bounded. Moreover, the constraint $L_{\|\tilde{x}\|^2-1 }^{(k)}[w^k] = 0$ gives that
\[
\langle \|\tilde{x}\|^{2t} ,w^k\rangle=\langle \|\tilde{x}\|^{2(t-1)} ,w^k\rangle=\cdots=\langle 1 ,w^k\rangle=(w^k)_0,
\]
which implies that  the diagonal elements
of $M_t[w^{k}|_{2 t}]$ are all bounded by  $(w^k)_0$. Combining with \reff{w:bound}, it implies that $\left\{w^{k}|_{2 t}\right\}_{k=t}^\infty$
is bounded. Suppose $w^*$ is an accumulation
point of $\left\{w^{k}|_{2 t}\right\}_{k=t}^\infty$ for $t\geq \frac{d_1}{2}$.  It is easy to see that $w^*$ is a minimizer of \reff{eq5.2}. When $t$ is sufficiently big, the flatness of $w^*$ can be proved as in  \cite[Theorem~5.3]{2014The}. We omit it for cleanness. Note that for generic $f$, the optimization problem $(\ref{eq5.2})$ has a unique minimizer. Hence, the truncation of each  accumulation point at degree $d$ must be an extreme point of \reff{extre}. Combining with Proposition~\ref{pro1},
we have that each of its accumulation points admits
a $r$-atomic $K$-representing measure with $r \leq m$.

(ii) Since $f$ is generic in $\Sigma[\tilde{x}]_{d_1}$,  the optimal values of $(\ref{eq5.2})$ and $(\ref{eq5.3})$ are equal. The condition \reff{finite} implies that the relaxation $\reff{eq5.4}-\reff{eq5.5}$ is tight for all $k\geq k_1$. For  $k\geq k_1$, the following holds
\[
\langle f- \sum\limits_{i = 1}^m  \theta_i^*\hat{a}_i,w^k\rangle=\langle f, w^k\rangle-\sum\limits_{i = 1}^m\theta_i^*\langle    \hat{a}_i,w^k\rangle \leq \langle f, w^k\rangle-\sum\limits_{i=1}^m b_i\theta_i^*=0.
\]
Note that $f - \sum\limits_{i=1}^m \theta_i^*\hat{a}_i
\in \mathscr{P}_{d_1}(\widetilde{K})$. It implies that the optimal value of \reff{lo:finite} is zero.
The $\ell$th order SOS relaxation for \reff{lo:finite} is
\be  \label{finite:sos}
\left\{ \baray{cl}
\max  & \gamma  \\
\st &f - \sum\limits_{i = 1}^m \theta_i^*\hat{a}_i -\gamma  \in
\ideal{\tilde{c}_{eq}}_{2\ell}+\qmod{\tilde{c}_{in}}_{2\ell}. \\
\earay \right.
\ee
The above implies that the relaxation \reff{finite:sos} is tight for all $\ell\geq k_1$. Hence, the assumptions in \cite{Nie2015Linear} for the problem  \reff{lo:finite} are satisfied and the truncation $\left.w^{ k}\right|_{2 t}$ is flat for all $t$, $k$ big enough.
\end{proof}

\section{Numerical Experiments}
\label{sc:num}

In this section, we give examples to
 solve truncated moment problems
with unbounded sets by Algorithm \ref{alg2}. In our computation, we set
\[
d_0 =  \big\lceil \frac{d+1}{2} \big\rceil, \quad
f = [x]_{d_0 }^{T} R^{T} R[x]_{d_0},
\]
where $R$ is a random square matrix obeying standard Gaussian distribution.
Algorithm \ref{alg2} is implemented in the software
GloptiPoly~3 \cite{2009GloptiPoly}
by calling the semidefinite program package Sedumi \cite{sturmusing}.
The computation is implemented in MATLAB R2019a on a Lenovo
Laptop with CPU@1.10GHz and RAM 16.0G.
For cleanness of the paper, only four digits are displayed
for computational results.

Recall that  a symmetric tensor
$\mathcal{B}=\left(\mathcal{B}_{i_{1} \ldots i_{m}}
\right)_{0 \leq i_{1}, \ldots, i_{m} \leq n}$
is uniquely determined by the tms
$\mathbf{b} \coloneqq (b_\af)_{\af \in \n^n_m}$ such that
\[
b_{\alpha} \, =  \, \mathcal{B}_{i_{1} i_{2} \ldots i_{m}}
\]
for every $\af \in \N^n_m$ with
$x_0^{m-|\alpha|}x^\alpha = x_{i_{1}}  \cdots x_{i_{m}}$.

First, we give some examples of PSOP tensor decompositions.

\begin{exm}
(i) Let $\mathcal{B} \in \mathrm{S}^{3}\left(\mathbb{R}^{5}\right)$
be the symmetric tensor such that
\[
  	\mathcal{B}_{i j k}=\left\{\begin{array}{l}
  		1 \text { if }  ijk=0, \\
  		3 \text { otherwise. }
  	\end{array}\right.
\]
We can see that $\mathcal{B} \in cl(\mathrm{PS}^{5}_3)\backslash\mathrm{PS}^{5}_3$,
i.e., the tms $\mathbf{b} \in cl(\mathscr{R}_{3}(\r^4))\backslash \mathscr{R}_{3}(\r^4)$. For each $\eps> 0$, let
\[
\mathbf{b}^{\epsilon}\, \coloneqq \,
2\epsilon^3[(\frac{1}{\epsilon},\frac{1}{\epsilon},\frac{1}{\epsilon},\frac{1}{\epsilon})]_{\n^4_3}
+[(1,1,1,1)]_{\n^4_3}.
\]
The membership $\mathbf{b} \in cl(\mathscr{R}_{3}(\r^4))$
follows from $\mathbf{b}^{\epsilon} \rightarrow \mathbf{b}$ as $\eps \to 0$.
To see  $\mathbf{b} \notin  \mathscr{R}_{3}(\r^4)$,
suppose otherwise that $\mathbf{b}$ admits a representing measure $\mu$. Then,
\[
\int (1-x_1)^2 d \mu=b_{(0,0,0,0)}-2b_{(1,0,0,0)}+b_{(2,0,0,0)}=0.
\]
This implies that $supp(\mu) \subseteq \{x_1=1\}$, so
\[
0=\int x_1(1-x_1)^2 d \mu=b_{(1,0,0,0)}-2b_{(2,0,0,0)}+b_{(3,0,0,0)}=2,
\]
which is a contradiction.
We apply the Algorithm \ref{alg2} to check $\mathbf{b} \in cl(\mathscr{R}_{3}(\r^4))$.
For each random instance of $f$, we get
$\operatorname{rank} M_{1}(w^k) = \operatorname{rank} M_{2}(w^k)=2$
at the relaxation order $k=2$ and obtain a $2$-atomic representing measure
whose support consists of the points
\begin{equation*}
( 0, \frac{1}{2},\frac{1}{2},\frac{1}{2},\frac{1}{2}), \quad
( \frac{1}{\sqrt{5}},  \frac{1}{\sqrt{5}},  \frac{1}{\sqrt{5}}, \frac{1}{\sqrt{5}},\frac{1}{\sqrt{5}}).
\end{equation*}
Note that $x_0$-coordinate of the atom $( 0, \frac{1}{2},\frac{1}{2},\frac{1}{2},\frac{1}{2})$ is  zero.
This  implies that $\mathcal{B} \in cl(\mathrm{PS}^{5}_3)$ by Remark~\ref{approa}.
The computation took about  $0.48$ second.

\noindent
(ii) Let $\mathcal{B} \in \mathrm{S}^{4}\left(\mathbb{R}^{6}\right)$
be the symmetric tensor such that
\[
\mathcal{B}_{i j k s}=i+k+j+s.
\]
By applying  Algorithm \ref{alg2}, the relaxation~\reff{eq5.4}
is infeasible at the order $k=3$.
Therefore, we get $\mathcal{B} \notin cl(\mathrm{PS}^{6}_4)$
and $\mathcal{B}$ is not PSOP. The computation took about $48.12$ seconds.

\noindent
(iii) Let  $\mathcal{B} \in \mathrm{S}^{3}\left(\mathbb{R}^{n+1}\right)$
be the symmetric tensor such that
\[
	\mathcal{B}_{i j k}=n-\max\{i,j,k\}.
\]
For the case $n=1,2,\dots,9$,  Algorithm~\ref{alg2} always produces
a PSOP decomposition for $\mathcal{B}$ at the order $k=2$.  For instance, when $n=5$, Algorithm~\ref{alg2} produces the following PSOP decomposition
\[
\mathcal{B}=
\bbm 1\\0\\0\\0\\0\\0 \ebm^{\otimes 3}+
\bbm 1\\1\\0\\0\\0\\0 \ebm^{\otimes 3}+
\bbm 1\\1\\1\\0\\0\\0 \ebm^{\otimes 3}+
\bbm 1\\1\\1\\1\\0\\0 \ebm^{\otimes 3}+
\bbm 1\\1\\1\\1\\1\\0 \ebm^{\otimes 3}.
\]
The computing time is presented in Table \ref{time1}.
Time consumption in the table is measured in seconds.
The computational time increases as the dimension increases.
This is because the matrix size of the relaxation~\reff{eq5.4}
for the order $k$ is around $O(n^k)$.

\begin{table}
	\label{time1}
\caption{Example 5.1 (iii)}
\begin{tabular}{l|c|c|c|c|c}
	\hline n & time & n & time& n & time \\
	\hline 1 & 0.08 &2 &  0.11 &3 &  0.21\\
	\hline 4&   0.43&5 &  0.85 &6 &  1.93 \\
	\hline 7& 4.48 &8 &   10.81&9 &25.95 \\
\hline
\end{tabular}
\end{table}

\noindent
(iv) Consider the tensor $\mathcal{B} \in \mathrm{S}^{6}\left(\mathbb{R}^{3}\right)$
such that $\mathbf{b} \coloneqq(b_\beta)_{\beta\leq \n^{2}_{6}}$
is (listed in degree-lexicographic order)
\[
\begin{array}{c}
\big( 1,1,0,1,0,1,1,0,1, c, 1,0,1, c, 1+c^{2}, 1,0,1, c, 1+c^{2}, 2 c+c^{3}, \\
    1,0,1, c, 1+c^{2}, 2 c+c^{3}, 1+3 c^{2}+c^{4}+t \big),
\end{array}
\]
for two parameters $c,t \in \r$.
This example is a variation of Example~5.2 of \cite{fialkow2013closure}.
As pointed in \cite{fialkow2013closure}, the tms $\mathbf{b}$
is flat and admits a representing measure when $t=0$
and $\mathbf{b} \in cl(\mathscr{R}_{6}(\r^2))\backslash \mathscr{R}_{6}(\r^2)$ when $t>0$.
For the case $t=c=0$, Algorithm~\ref{alg2} produces the PSOP decomposition
\[
\mathcal{B} \, =  \,
\frac{1}{2} \left[\baray{r} 1 \\ 1 \\ -1 \earay\right]^{\otimes 6}
+ \frac{1}{2} \left[\baray{r} 1 \\ 1 \\ 1 \earay\right]^{\otimes 6}
\]
at the relaxation order $k=4$. It took around $0.97$ second.
Next we consider the case $t=1$, $c=0$. For each random instance of $f$,
we  get $\operatorname{rank} M_{1}(w^k)=\operatorname{rank} M_{2}(w^k)=3$
at the order $k=4$ and it produces a $3$-atomic measure $\mu$
for the homogenization $\widetilde{\mathbf{b}}$.
For some random instances of $f$, the support $\supp{\mu}$ consists of the points
\begin{equation*}
( \frac{1}{\sqrt{3}}, \frac{1}{\sqrt{3}},- \frac{1}{\sqrt{3}}), \quad
( \frac{1}{\sqrt{3}},\frac{1}{\sqrt{3}}, \frac{1}{\sqrt{3}}),
\quad  (0,0,-1).
\end{equation*}
For other instances of $f$, the support $\supp{\mu}$ consists of the points
\begin{equation*}
( \frac{1}{\sqrt{3}}, \frac{1}{\sqrt{3}},- \frac{1}{\sqrt{3}}), \quad
( \frac{1}{\sqrt{3}},\frac{1}{\sqrt{3}}, \frac{1}{\sqrt{3}}),
\quad  (0,0,1).
\end{equation*}
Since $x_0$-coordinates of $(0,0,1)$ and $(0,0,-1)$ are zero,
we get $\mathcal{B} \in cl(\mathrm{PS}^{3}_6)$.
The computation took about $0.96$ second.
\end{exm}

In the following, 	we give some examples of SCP tensor decompositions.
Recall that the tensor $\mathcal{B}$ is SCP
if and only if $\mathbf{b}$ admits a representing measure supported in
the nonnegative orthant.

\begin{exm}
Consider the symmetric tensor $\mathcal{B}(t) \in \mathrm{S}^{3}\left(\mathbb{R}^{6}\right)$
such that
\[
	\mathcal{B}(t) \, =  \,
\left[\baray{r} 1 \\ 1 \\ 0 \\0\\0 \\0\\ \earay\right]^{\otimes 3}
	+  \left[\baray{r} 1 \\ 0 \\ 1\\0\\0\\0\\ \earay\right]^{\otimes 3}+
\left[\baray{r} 1 \\ 0 \\ 0\\1\\0\\0\\ \earay\right]^{\otimes 3}+
\left[\baray{r} 1 \\ 0 \\ 0\\0\\1\\0\\ \earay\right]^{\otimes 3}+
\left[\baray{r} t \\ 0 \\ 0\\0\\0\\1\\ \earay\right]^{\otimes 3}
\]
for a parameter $t\in \r$.
For $t>0$, we clearly have $\mathcal{B}(t) \in \mathrm{SCP}_3^6$.
However, for $t=0$, we have
$\mathcal{B}(0) \in cl(\mathrm{SCP}_3^6) \backslash \mathrm{SCP}_3^6$.
To see this, note that
\[
\mathcal{B}(t)_{555} = 1, \quad
\mathcal{B}(t)_{055} = t.
\]
If $\mathcal{B}(0)$ was otherwise SCP,
then there would exist a Borel measure $\mu$
such that $\supp{\mu} \subseteq \re^6$ and
\[
\int x_5^3\mathrm{d}\mu = \mathcal{B}(0)_{555} = 1,
\int x_5^2\mathrm{d}\mu = \mathcal{B}(0)_{055} = 0.
\]
But such a Borel measure $\mu$  cannot exist, so we get
$\mathcal{B}(0) \not\in \mathrm{SCP}_3^6$. Moreover,
since $\mathcal{B}(t) \to \mathcal{B}(0)$ as $t \to 0$ and $t > 0$,
we get $\mathcal{B}(0) \in cl(\mathrm{SCP}_3^6)$. Similarly,
we can show that $\mathcal{B}(t) \notin \mathrm{SCP}_3^6$ for $t<0$.
The following is the performance of Algorithm \ref{alg2}
for checking SCP tensors.
	
(i) Consider the case $t=1$. By  applying Algorithm \ref{alg2}, it produces the decomposition
$
\mathcal{B}(1) = \sum_{k=1}^{5} \rho_{k} \tilde{u}_k^{\otimes 3}
$
at the relaxation order $k=2$, where
\[
\baray{ll}
\tilde{u}_1=\frac{1}{\sqrt{2}}(1,1,0,0,0,0),&\rho_{1}= 2\sqrt{2};\\
\tilde{u}_2=\frac{1}{\sqrt{2}}(1,0,1,0,0,0),&\rho_{2}= 2\sqrt{2};\\
\tilde{u}_3=\frac{1}{\sqrt{2}}(1,0,0,1,0,0),&\rho_{3}= 2\sqrt{2};\\
\tilde{u}_4=\frac{1}{\sqrt{2}}(1,0,0,0,1,0),&\rho_{4}= 2\sqrt{2};\\
\tilde{u}_5=\frac{1}{\sqrt{2}}(1,0,0,0,0,1),&\rho_{5}= 2\sqrt{2}.
\earay
\]
It took about $ 0.97$ second.
Since all the $x_0$-coordinates of $\tilde{u}_k$ are positive,
we know that $\mathcal{B}(1)$ is an SCP tensor.

(ii) Consider the case $t=0$.  By  applying Algorithm \ref{alg2},
it produces the decomposition $\mathcal{B}(0) = \sum_{k=1}^{5}\rho_{k}
\tilde{u}_k ^{\otimes 3}$ at the order $k=2$, where
\[
\baray{ll}
\tilde{u}_1=\quad \, \, (0,0,0,0,0,1),&\rho_{1}= 1;\\
\tilde{u}_2=\frac{1}{\sqrt{2}}(1,1,0,0,0,0),&\rho_{2}=2\sqrt{2};\\
\tilde{u}_3=\frac{1}{\sqrt{2}}(1,0,1,0,0,0),&\rho_{3}=2\sqrt{2};\\
\tilde{u}_4=\frac{1}{\sqrt{2}}(1,0,0,1,0,0),&\rho_{4}= 2\sqrt{2};\\
\tilde{u}_5=\frac{1}{\sqrt{2}}(1,0,0,0,1,0),&\rho_{5}=2\sqrt{2}.\\
\earay
\]
Since  $x_0$-coordinate of $\tilde{u}_1$ is zero,
we  get $\mathcal{B}(0) \in cl(\mathrm{SCP}_3^6)$.
The computation took about $1.09$ seconds.

(iii) Consider the case $t=-1$. By  applying Algorithm \ref{alg2},
the relaxation \reff{eq5.4} is infeasible at the relaxation order $k=2$,
so we get $\mathcal{B}(-1) \notin \mathrm{SCP}_3^6$. 	
The computation took about $0.82$ second.
\end{exm}

In the following, we present some examples on the generalized truncated moment
problems with unbounded sets.

\begin{exm}
\noindent
(i)	Consider $K=\{x \in \r^6 : \|x\|^2 \geq 1 \}$,
$y \in \r^{\a}$ is listed as follows
\[
	y_{(0,0,0,0,0,0)}=t,~~y_{(2,0,0,0,0,0)}=\cdots=y_{(0,0,0,0,0,2)}=1.
\]
When $t> 6$, $y$ admits no $K$-representing measures.
This is because if $\mu$ is such a measure, we have
\[
6=\int_K x_1^2+\cdots+x_6^2 ~\mathrm{d} \mu \geq \int_K 1~ \mathrm{d} \mu =t,
\]
which is a contradiction. And if $t=6$, the measure $\nu=6\delta_{(\frac{1}{\sqrt{6}},\cdots, \frac{1}{\sqrt{6}})}$ is a $K$-representing measure for $y$. Hence, $y\in \rk$ when $t=6$.

Consider $t=6$, and we apply the Algorithm \ref{alg2}.
At the order $k=2$, we have that $\operatorname{rank} M_{1}(w^k)=\operatorname{rank} M_{2}(w^k)=1$
and obtain a $K$-representing measure
$6\delta_{(\frac{1}{\sqrt{6}},\cdots, \frac{1}{\sqrt{6}})}$ of $y$.
The computation took about $5.93$ seconds.
Next, we consider $t=7$.  At the order $k=2$, the relaxation \reff{eq5.4} is infeasible,
which gives a certificate that $y$ admits no $K$-representing measures when $t=7$.
It took about $ 4.23$ seconds.

\noindent
(ii) Consider $K=\{x \in \r^3 \mid  x_1x_2x_3(x_1+x_2+x_3-6)=0,\,x_i\geq 0,\,i=1,2,3\}$,
the index set $\a$ and $y \in \r^{\a}$ are listed as follows  	
\[
\begin{array}{cc|cc|cc}
	\alpha & y_{\alpha} & \alpha & y_{\alpha}  & \alpha & y_{\alpha}  \\
	\hline
	(1, 0, 0)&9&(0, 1 ,0)&15&(0 , 0 ,1)&18\\
	\hline
	(2, 0, 0)&9&(1, 1 ,0)&14&(0 , 2 ,0)&29\\	
	\hline
	(1, 1, 1)&24&(1, 2 ,0)&28&(1 , 0 ,2)&44\\	
\end{array}
\]
By applying Algorithm~\ref{alg2}, we get $\operatorname{rank} M_{1}(w^k)=\operatorname{rank} M_{3}(w^k)=5$
at the order $k=3$. We get a $5$-atomic  $K$-representing measure $\mu=\sum\limits_{i=1}^5 \lambda_{i}\delta_{u_i}$ of $y$ as follows
\[
\baray{ll}
u_1=( 4.8367 ,  0.0000 ,  -8.6952
),&\lambda_{1}= 0.0535;\\
u_2=( 0.0000 ,  -0.4954  ,  0.6272),&\lambda_{2}=4.4833;\\
u_3=( 0.0898 ,0.0000,    0.7193),&\lambda_{3}=11.5641;\\
u_4=( 1.4278 ,   2.2915  ,  2.2807),&\lambda_{4}=3.2163;\\
u_5=( 0.3530  ,  1.1177 ,  0.0000),&\lambda_{5}=  8.8139.\\
\earay
\]
The computation took about  $0.24$ second.

   	
\noindent
(iii) Consider $K=\{x \in \r^6 \mid x_i^2-x_ix_{i+1}+1=0,~i=1,\dots,5 \}$,
the index set $\a$ and $y \in \r^{\a}$ are listed as follows
\[
\begin{array}{cc|cc}
	\alpha & y_{\alpha} & \alpha & y_{\alpha}  \\
	\hline
	(4, 0, 0,0,0, 0)&1&(0, 4, 0,0,0,0)&3\\
	\hline (0, 0, 4,0,0,0)&6&(0, 0, 0,4,0,0)&10\\
	\hline (0, 0, 0,0,4,0)& 15&( 0, 0,0,0,0, 4)&21\\	
\end{array}
\]
Note that $y$ admits no $K$-representing measures.
Suppose otherwise $\mu$ is a $K$-representing measure of $y$, then we have
\[
3=\int_K x_2^4\mathrm{~d\mu}=\int_K (x_1+\frac{1}{x_1})^4\mathrm{~d\mu}\geq \int_K x_1^4+6\mathrm{~d\mu}=7,
\]
which is a contradiction. By applying  Algorithm \ref{alg2},
we have $\operatorname{rank} M_{2}(w^k)=\operatorname{rank} M_{3}(w^k)=6$
at the order $k=3$. The atoms of the finitely atomic measure $\nu$ for $w^{ *}|_{6}$ are
\[
\begin{array}{rr} \frac{1}{\sqrt{6}}(0, 1, 1, 1, 1, 1, 1),&
\frac{1}{\sqrt{5}}(0,0, 1, 1, 1, 1, 1),\\
\frac{1}{2} (0,0,0, 1, 1, 1, 1), &
\frac{1}{\sqrt{3}}(0,0,0,0, 1, 1, 1),\\
\frac{1}{\sqrt{2}}  (0,0,0,0,0, 1, 1), & (0,0,0,0,0,0,1).\\	
\end{array}
\]
Thus, we know that $y\in cl(\rk)$.
The computation took about $13.88$ seconds. \\
(iv) Let  $K=\mathbb{R}^{6}$.
One wonders whether there exists a measure $\mu$ supported in $K$ such that
\[
\int x_{i}^{2}  \mathrm{d} \mu=i\,(i=1,\dots,6), \int x_{i}x_{i+1} \mathrm{d} \mu=i^2\,(i=1,\dots,5),\,\int x_{1}x_{6} \mathrm{d} \mu=1.
\]	
By applying Algorithm \ref{alg2},
the relaxation $(\ref{eq5.4})$ is infeasible at the order $k=2$,
which gives a certificate that the system above is infeasible.
The computation took about $1.56$ seconds.


\noindent
(v)	Let  $K=\{x \in \mathbb{R}^{4} \mid \|x\|^2 \geq 1\}$.
We want to know whether there exists a measure $\mu$ supported in $K$ such that
	\[
		\int x_{1}^{3} x_{2}^{3} \mathrm{d} \mu =\int x_{2}^{3} x_{3}^{3} \mathrm{d} \mu=\int x_{3}^{3} x_{4}^{3} \mathrm{d} \mu=\int x_{4}^{3} x_{1}^{3} \mathrm{d} \mu,
	\]
	 \[
	 \int\left(x_{1}^{4} x_{2}^{2}+x_{2}^{4} x_{3}^{2}+x_{3}^{4} x_{4}^{2}+x_{4}^{4} x_{1}^{2}\right) \mathrm{d} \mu\geq4,\, \int\left(x_{1}^{6} +x_{2}^{6} +x_{3}^{6} +x_{4}^{6}\right) \mathrm{d} \mu\geq4,
	 \]
	 \[
	 \int (x_{1}^{2} x_{2}^{2} x_{3}^{2}+x_{2}^{2} x_{3}^{2} x_{4}^{2}+x_{3}^{2} x_{4}^{2} x_{1}^{2}+x_{4}^{2} x_{1}^{2} x_{2}^{2}) \mathrm{d} \mu\leq4.
	 \]		
By applying   Algorithm \ref{alg2}, we have that
$\operatorname{rank} M_{1}(w^k)=\operatorname{rank} M_{2}(w^k)=4$
at the order $k=4$, and the $x_0$-coordinates of all atoms are nonzero.
Thus, we get a $4$-atomic $K$-representing measure $\mu=\sum\limits_{i=1}^4 \lambda_{i}\delta_{u_i}$ of $y$ as follows
\[
\baray{ll}
u_1=( -0.1401 ,  -1.8084,   -1.3964,   -0.1014
),&\lambda_{1}= 0.0460;\\
u_2=( -1.8723  , -1.3077,   -0.1776  ,  0.0536),&\lambda_{2}= 0.0504;\\
u_3=(-1.4048 ,  -0.0263  , -0.0682  , -1.8086),&\lambda_{3}= 0.0452;\\
u_4=(-0.0154 ,  -0.0324  ,  1.8546 ,   1.3451),&\lambda_{4}=0.0477.\\
\earay
\]
The computation took about $320.71$ seconds.

\end{exm}

We would like to point out that Algorithm \ref{alg2} can be directly applied
to solve moment optimization problems  with unbounded sets of type \reff{eq5.2}.
This is discussed in the work \cite{2008A,Nie2015Linear}.
We give such an example of rational polynomial optimization.
Consider the rational polynomial optimization
\be \label{rational}
\left\{ \baray{rl}
\min & \frac{f(x)}{g(x)} \\
 \st  &   x \in K,
\earay  \right.
\ee
where  $f$, $g$ are polynomials. Here $K$ is as in \reff{1.1} and
$g(x)$ is assumed to be nonnegative on $K$.
It can be shown that \reff{rational} is equivalent to
\be \label{rationpol}
\left\{ \baray{rl}
 \min  & \langle f,y\rangle \\
 \st & \langle g,y\rangle=1 ,\\
  & y\in \mathscr{R}_d(K),
\earay  \right.
\ee
where $d=\max\{\deg(f),\deg(g)\}$.
Note that \reff{rationpol} is a generalized problem of moments.
Let $\hat{f}$, $\hat{g}$ be the degree-d homogenization of $f$, $g$ respectively. When $K$ is closed at $\infty$,
one can see that \reff{rationpol} is equivalent to
\be  \label{rationhom}
\left\{ \baray{rl}
\min &\langle \hat{f},w\rangle \\
\st  & \langle \hat{g},w\rangle=1 ,\\
&  w\in \mathscr{R}_d(\widetilde{K}) .
\earay  \right.
\ee
Algorithm~\ref{alg2} can be similarly applied to solve it.
The following is such an example.
\begin{exm}
Consider the problem
\[
\min_{x_1,x_2\geq0} \quad \frac{ x_{1}^{3}+x_{2}^{3}+3 x_{1} x_{2}+1 }
{ x_{1}\left(x_{2}^{2}+1\right)+x_{2}\left(x_{1}^{2}+1\right)+
\left(x_{1}^{2}+x_{2}^{2}\right)}.
\]
In this example,
\[
f=x_{1}^{3}+x_{2}^{3}+3 x_{1} x_{2}+1,  \quad
g=x_{1}\left(x_{2}^{2}+1\right)+x_{2}\left(x_{1}^{2}+1\right)+
\left(x_{1}^{2}+x_{2}^{2}\right).
\]
The optimal value is 1. We apply  the Moment-SOS relaxation \reff{eq5.4}-\reff{eq5.5}
to solve  \reff{rationhom}. At the order $k=2$,
the flat truncation \reff{flat:condi} is satisfied,
and the computed optimal value is $1$.
\end{exm}

\section{Conclusions and discussions}
\label{sc:con}

In this paper, we study the geometric properties of the moment cone
and the cone of nonnegative polynomials. Semidefinite relaxations are constructed
to approximate $\rk$ and $\pk$, and the convergence is
guaranteed under the assumptions that
$K$ is closed at $\infty$ and $\r[x]_{\a}\cap int(\mathscr{P}_{d}(K)) \ne \emptyset$.
Based on these relaxations, efficient algorithms are given for solving
the generalized truncated moment problem.
Note that we do not need  $K$ to be compact as the previous work,
while asymptotic  and finite convergence is guaranteed.
There are broad applications of generalized
moment optimization problems (see \cite{de,hny21, 2008A,Nie2015Linear}).

There is still much interesting future work to do.
How can one certificate whether or not
a multi-sequence $y\in \r^{\a}$ admits a $K$-representing measure
for the case that $y$ lies on the boundary of  $\mathscr{R}_{\a}(K)$?
When $\pk \cap int(\mathscr{P}_{d}(K)) = \emptyset $,
do we still have the same properties for the Moment-SOS relaxations?
How can we characterize the relative interior of $\pk$ when
it lies on the boundary of $\mathscr{P}_{d}(K)$?
These questions are mostly open for us. Last, if  $y\in \r^{\a}$
admits a $K$-representing measure, how can we get a finitely atomic measure with minimum length?
This problem is quite important for detecting different ranks of matrices and tensors.

\textbf{Acknowledgements.}
	Lei Huang and Ya-Xiang Yuan are partially supported by 
	the National Natural Science Foundation of China (No.~12288201).



\end{document}